\documentclass[10pt]{amsart}
\usepackage{graphicx,amssymb,amsfonts,amsmath,amsthm,newlfont}
\usepackage{epsfig,url}
\usepackage{color}
\usepackage{enumitem}

\usepackage[all,2cell]{xy} \UseAllTwocells \SilentMatrices

\newtheorem{thm}{Theorem}[section]
\newtheorem{problem}{Problem}[section]
\newtheorem{cor}[thm]{Corollary}
\newtheorem{lem}[thm]{Lemma}
\newtheorem{prop}[thm]{Proposition}
\theoremstyle{definition}
\newtheorem{defn}[thm]{Definition}

\newtheorem{example}[thm]{Example}

\theoremstyle{remark}
\newtheorem{rem}[thm]{Remark}
\numberwithin{equation}{section}

\newcommand{\Z}{\mathbb Z}
\newcommand{\C}{\mathbb C}

\newcommand{\R}{\mathbb R}
\newcommand{\N}{\mathbb N}
\newcommand{\Pro}{\mathbb P}

\newcommand{\I}{\mathbb{I}}

\newcommand{\gr}{\mathrm{gr}}

\newcommand{\MT}{\mathcal{MT}}

\newcommand{\Dom}{\mathrm{S}}

\newcommand{\PGL}{\mathrm{PGL}}

\newcommand{\av}{\mathsf{a}}
\newcommand{\bv}{\mathsf{b}}

\newcommand{\zetam}{\zeta^{ \mathfrak{m}}}
\newcommand{\zetau}{\zeta^{ \mathfrak{u}}}

\newcommand{\Q}{\mathbb Q}

\newcommand{\To}{\longrightarrow}

\newcommand{\A}{\mathbb{A}}

\newcommand{\Or}{\mathcal{O}}

\newcommand{\mm}{\mathfrak{m} }

\newcommand{\uu}{\mathfrak{u} }

\newcommand{\rel}{\mathrm{rel}}

\newcommand{\Mod}{\mathcal{M}}

\newcommand{\ord}{\mathrm{ord}}

\begin{document}
\author{Francis Brown}
\begin{title}[Irrationality proofs for zeta values, moduli spaces and dinner parties]{Irrationality proofs for zeta values, moduli spaces and dinner parties}\end{title}
\maketitle

 \begin{abstract}
 A simple  geometric construction on the moduli spaces $\Mod_{0,n}$ of curves of genus $0$ with $n$
  ordered marked points  is described which gives a common framework for  many  irrationality proofs for zeta values. 
 This construction yields Ap\'ery's approximations to $\zeta(2)$ and $\zeta(3)$, and for larger $n$,  
an infinite family of small linear forms in  multiple zeta values with an interesting algebraic structure. 
It also contains a generalisation of
 the linear forms used by Ball and Rivoal  to prove that infinitely many odd zeta values are irrational. 
   \end{abstract}

 \section{Introduction}
 \subsection{Summary} A folklore conjecture states that the values of the Riemann zeta function at odd integers $\zeta(3), \zeta(5), \ldots$ and $\pi$ 
 are algebraically independent over  $\Q$.   Very little is known about this conjecture, except for the following remarkable facts:
 \begin{enumerate}
 \item That $\pi$ is transcendental, proved by Lindemann in 1882.
 \item That $\zeta(3)$ is irrational, proved by  Ap\'ery in 1978.
 \item  That the  $\Q$-vector space spanned by the odd zeta values $\zeta(3), \zeta(5),\ldots$ is infinite dimensional (proved by Ball and   Rivoal \cite{RiCRAS,  B-R}).
 \item  That at least one amongst  $\zeta(5),\zeta(7),\zeta(9),\zeta(11)$ is irrational (Zudilin \cite{Zu}).
 \end{enumerate}
The irrationality of $\zeta(5)$, or $\zeta(3)/\pi^3$,   are  open problems.
 All of the  above results can  be proved by  constructing small linear forms in zeta values 
using elementary  integrals. Quantitive results, such as bounds on the irrationality
 measures of $\zeta(2)$   and $\zeta(3)$   \cite{RV1,RV2}, and bounds on the transcendence measure of $\pi^2$  \cite{Sorokin},
 can also be obtained by similar methods.
 
 The starting point for this paper is the observation that the integrals in all these proofs are equivalent, after a suitable change of  variables,  to  period integrals on the moduli space  $\Mod_{0,n}$ of 
 curves of genus zero with $n$  marked points.   Conversely, we know by \cite{BrENS}, or \cite{GM} together with \cite{BrMTZ},
 that all period integrals on $\Mod_{0,n}$ are linear forms in multiple zeta values. This provides a huge family  of potential candidates for generalising the above results. Unfortunately, the typical period integral involves all multiple zeta values up to weight $n-3$ and is ill-adapted  for an irrationality proof. 
 In this paper, we describe a  narrower  class of period integrals on $\Mod_{0,n}$, based on a variant of  the classical dinner table problem  \cite{Dinner, Poulet},
in which certain multiple zeta values vanish. We show that this restricted family of integrals  has some special properties, and reproduces most, and possibly all, the results alluded to in the first paragraph.

 \subsection{Structure of irrationality proofs} \label{sectStruct}
  The basis for the above results is the construction of small linear forms in zeta values.
 More generally,  suppose that we have:
 \begin{enumerate}[itemsep=2pt,parsep=0pt]
 \item For all $n \geq 0$, a non-zero  $\Q$-linear combination
 $$ I_n = a^{(1)}_n \zeta_1 + \ldots + a^{(k)}_n \zeta_k\ , $$ 
 where $a^{(i)}_n  \in \Q$, and $\zeta_1, \ldots, \zeta_k$ are  fixed  multiple zeta values. 
 \item A bound on the linear forms $I_n$. For example, they satisfy an  inequality 
 $$0 < |I_n| < \varepsilon^n$$
 for all $n \geq 0$, where $\varepsilon $ is a small positive real number.
 \item Some control on the coefficients $a^{(i)}_j$.  At its most basic, 
 this is simply a bound on the denominators of $a^{(k)}_n$ as a function of $n$. This is often a function of 
 $$ d_n = \mathrm{lcm}\{1,\ldots, n\} $$
 The prime number theorem implies that $\lim_{n\rightarrow \infty} d_n^{1/n} =e.$
  \end{enumerate}
 Only in very specific cases, when the bounds $(3)$ on the coefficients are favourable compared to the constant $\varepsilon$ in  $(2)$, 
  can one deduce irrationality results.    For Ap\'ery's theorem, $k=2$ and one constructs linear forms $I_n=a_n \zeta(3)+ b_n$,  for example, as integrals $(\ref{AIBeuk3})$. 
  In this case, a   bound on the denominators of $a_n, b_n$ suffices:
  we have 
  $$ \varepsilon = (\sqrt{2}-1)^4 \qquad \hbox{ and }  \qquad a_n \in \Z, \quad  d_n^3 b_n \in \Z$$
   and the inequality 
   $$ e^3 \varepsilon = 0.591\ldots  <1$$
   is enough to deduce the irrationality of $\zeta(3)$.   For Ball and Rivoal's theorem, one constructs linear forms in odd zeta values $\zeta(2m+1)$ and applies a criterion due to Nesterenko \cite{N} which depends on the size of  both the denominators and the numerators of the coefficients  $a^{(i)}_j$
 to deduce a lower bound for
 $\dim_{\Q} \langle \zeta_1,\ldots, \zeta_k\rangle_{\Q}$.
 \footnote{Condition $(2)$ must  be slightly modified: one can assume by clearing denominators that $I_n$ has integer coefficients, and one needs to know that $|I_n|^{1/n}$ has a small positive limit  as $n\rightarrow \infty$.}

   Unfortunately, there are very few cases where this works, which motivates trying to reach a better understanding 
 of the general principles involved. We refer to  Fischler's Bourbaki talk for an excellent survey of known results \cite{FiBourbaki}.

 \subsection{Periods of moduli spaces $\Mod_{0,n}$}
A large supply  of linear forms  satisfying $(1)-(3)$ comes from period integrals on moduli spaces.
  Let $n\geq 3$ and let $\Mod_{0,n}$ denote the moduli space of curves of genus zero with $n$ ordered marked points.  It is isomorphic to the complement  in affine space $\A^\ell$, where $\ell=n-3$, of a hyperplane configuration
 $$\Mod_{0,n} = \{(t_1,\ldots, t_{\ell})  \in \A^{\ell}:  t_i \neq t_j , t_i \neq 0,1 \}\ .$$
  A connected component of $\Mod_{0,n}(\R)$ is given by the simplex
  $$\Dom_n = \{(t_1,\ldots, t_{\ell} )\in \R^{\ell}:  0 <  t_1 < \ldots < t_{\ell} <  1\}\ .$$
 Examples of  period integrals on $\Mod_{0,n}$ can be expressed as 
 \begin{equation} \label{IntM0n} \int_{\Dom_n}  \prod t_i^{a_i} (1-t_j)^{b_j}(t_i-t_j)^{c_{ij}} dt_1 \ldots dt_{\ell}
 \end{equation}
 for suitable $a_i, b_j, c_{ij} \in \Z$  such that the integral converges. 
  For such a family of integrals, the first property $(1)$ is guaranteed by the following  theorem:
 
 \begin{thm} \label{thmperiodsofmon} The periods of moduli spaces $\Mod_{0,n}$ are $\Q[2\pi i]$- linear combinations of multiple zeta values 
  of total weight $\leq \ell$.
  \end{thm}
   A general recipe  for constructing linear forms in multiple zeta values is to consider a family of  convergent  integrals
  \begin{equation}  \label{INdef}
   I_{f,\omega}(N) = \int_{\Dom_n} f^N \omega 
  \end{equation}
 where $\omega \in \Omega^{\ell} (\Mod_{0,n}; \Q)$ is a regular $\ell$-form, and $f \in \Omega^0(\Mod_{0,n};\Q)$. If, furthermore, one  imposes that  the rational function $f$ has zeros   along the boundary\footnote{or more precisely,
 the boundary of the inverse image of  $\Dom_n$ in the Deligne-Mumford-Knudsen compactification $\overline{\Mod}_{0,n}$ of $\Mod_{0,n}$}
  of   $\Dom_n$,
   then the integrals $I_{\ell}$ will be small, and condition  $(2)$ will automatically  hold as well, for some  small $\varepsilon$. 
  The proof of theorem  $\ref{thmperiodsofmon}$ given in \cite{BrENS} is   effective and should  in principle yield explicit bounds on the denominators (and numerators) of the  rational coefficients $a^{(i)}_j$ as a function of the order of the poles of the integrand. Furthermore, 
  
\begin{prop} All diophantine constructions  mentioned above  (with the possible exception of Zudilin's theorem $(4)$) can be expressed as   integrals of the type $(\ref{INdef})$.
\end{prop}

 The proof of this proposition uses results due to Fischler to convert the integrals listed in Appendix $1$ into a form equivalent to $(\ref{INdef})$.
 Nonetheless, finding good linear forms in zeta values amongst the integrals $(\ref{INdef})$ is significantly harder than finding  a needle in a haystack. 
 For example, the general integral yielding linear forms in multiple zeta values of weight at most 5 (of interest, if one seeks linear forms in $1$ and $\zeta(5)$)
 depends on 20 independent parameters, which is hopelessly large.
 
 \subsection{Vanishing of coefficients} 
   Therefore examples  such as $(\ref{INdef})$  provide an enormous supply of candidates $I_N$ for irrationality proofs.  
   The problem with this approach is  that the linear forms $I_N$ involved are rather weak, and only enable one  to deduce linear  independence of
a small fraction of the numbers $\zeta_i$. Furthermore, the generic integral $(\ref{IntM0n})$ contains
all multiple zeta values of weight up to and including $\ell$.  Thus, the presence of terms such as $\zeta(2n)$, for example, for which one already knows the linear independence by Lindemann's theorem,  blocks any further progress. 
\vspace{0.1in}

One  therefore requires, in addition to $(1)-(3)$ above:
 
\begin{enumerate}[itemsep=2pt,parsep=0pt]
\setcounter{enumi}{3}
 \item Vanishing theorems for some of the coefficients $a^{(i)}_j$. 
 \end{enumerate}
 \vspace{0.1in}
 
 This is already clear in the case of Ap\'ery's proof for $\zeta(3)$. Indeed, the generic period integral on $\Mod_{0,6}$
 gives rise to linear forms in $1$, $\zeta(2)$ and $\zeta(3)$, and a naive attempt at constructing linear forms $I_N$ only gives back 
 a proof that one of the two numbers $\zeta(2)$ and $\zeta(3)$ is irrational. The entire difficulty is thus to find integrals $I_N$ for which
 the coefficient of $\zeta(2)$ always vanishes (without destroying properties $(1)-(3)$). The key insight of Ball and Rivoal's proof, likewise, is the use of very well-poised hypergeometric
 series to construct linear forms in odd zeta values, and \emph{odd zeta values only}.

The vanishing problem $(4)$ can be rephrased in terms of algebraic geometry, and more precisely, the   cohomology of moduli spaces. In principle, this part of the problem is purely combinatorial.
An integral of the form 
$(\ref{INdef})$ can be expressed as a period of a certain relative cohomology group first introduced in \cite{GM}
$$m (A,B) = H^{\ell} (\overline{\Mod}_{0,n} \backslash A , B \backslash (B \cap  A))$$
  where $A, B$ are boundary divisors on the Deligne-Mumford compactification $\overline{\Mod}_{0,n}$. The divisor $A$ is determined from the singularities of the integrand, and  $B$ contains the  boundary of the closure of the  domain of integration.
 It was shown  in \cite{GM} that $m(A,B)$ is a mixed Tate motive over $\Z$ (which, by \cite{BrMTZ}, gives 
 another proof that its periods are multiple zeta values). Its de Rham realisation  $m (A,B)_{dR}$ is  a finite-dimensional  $\Q$-vector space graded in even degrees,  and a naive  observation (theorem \ref{thmCohomVanish}) is that 
 $$\gr^W_{2k} m(A,B)_{dR} =0  \qquad \Rightarrow \qquad \hbox{ vanishing of coefficients } a^{(i)}_j \hbox{ in weight } k\ . $$
This gives a sufficient condition for all multiple zeta values of weight $k$ to disappear.\footnote{It is not a necessary condition:  there could be more subtle reasons for the vanishing
of coefficients $a^{(i)}_j$. For example, the action of a  group of symmetries on the $m(A,B)_{dR}$ together with 
representation-theoretic arguments might give more powerful vanishing criteria.
}
The dimensions of the graded weight pieces $\gr^W_{2k} m(A,B)_{dR}$ can be computed from the data of the divisors $A,B$, and so reduces to a (rather tricky) combinatorial problem. I expect that the recent work of Dupont  \cite{Du1, Du2} may shed light on how to understand the vanishing problem $(4)$ from this viewpoint.

As a final remark, an irreducible boundary divisor  $D$ occurs in  $A$ if and only if a certain  linear form $\ell_D$ in the exponents $a_i,b_j,c_{ij}$ of the integrand $(\ref{thmperiodsofmon})$  is negative. Thus  the  flow of information goes as follows:
\begin{multline} \label{philosophy}
\hbox{ linear inequalities } \ell_D \leq  0 \hbox{ in the exponents   of  }  (\ref{thmperiodsofmon}) \\
\To \hbox{ vanishing  of  certain components } \gr^W_{2k} m(A,B)_{dR} \\
\To  \hbox{ vanishing of coefficients } a^{(i)}_j \qquad \qquad 
\end{multline}
The challenge is to make this  philosophy work, or failing that, to show  that one cannot construct moduli space motives  $m(A,B)$
with arbitrary vanishing properties.
Motivated by $(\ref{philosophy})$,  I was only able to find a general method to  force the coefficients of sub-maximal weight $2\ell-2$ to vanish (corresponding to MZV's of weight $\ell-1$), via the following construction.

 \subsection{A variant of the dinner table problem  \cite{Dinner}}
  Suppose that we have $n$ guests for dinner, sitting at a round table. Since it could  be  boring to talk to the same person for the whole duration of the meal, the guests should be  permuted   after the main course in such a way that no-one is sitting next to someone they previously sat next to. We can represent the new seating arrangement (non-uniquely) by a permutation $\sigma$ on $\{1,\ldots, n\}$,
which we write as $(\sigma(1), \ldots, \sigma(n))$.
 The  number of dinner table arrangements was first computed by Poulet in 1919 \cite{Poulet}.

 \begin{figure}[h]
{\includegraphics[height=3cm]{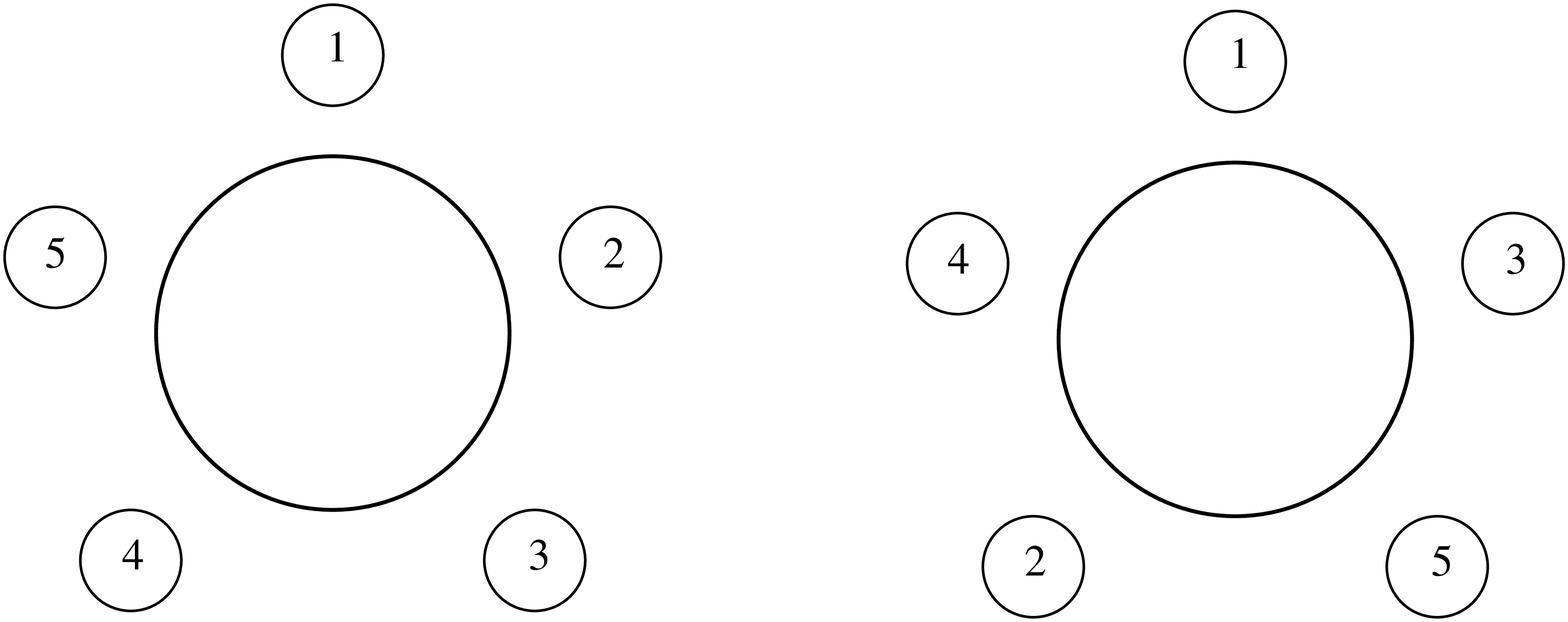}} 
\put(-174,37){\small $\delta^0$}
\put(-52,37){\small $\sigma \delta^0$}
\caption{The first solution for the classical dinner table problem is for $n=5$, and is unique up to symmetries. On the left is the original seating plan of guests;
on the right, the new arrangement after applying a permutation $\sigma$. No two neighbours are consecutive.}
\end{figure}

We need the following variant. Let $\delta^0$ denote the standard circular arrangement on $\{1,\ldots, n\}$ given by the integers modulo $n$ (the initial seating plan), 
and let $\sigma $ be any permutation on $\{1,\ldots, n\}$ (the new seating plan). Call a permutation $\sigma$  \emph{convergent} 
if no set of $k$ elements in $\{1,\ldots, n\}$ are simultaneously consecutive for $\delta^0$ and $\sigma \delta^0$, for all $2 \leq k \leq n-2$. 
For $n\leq 7$ this is equivalent to the classical dinner table problem but  for $n\geq 8$ this  imposes a genuinely new condition. The figure below illustrates a seating arrangement
 $\sigma= (2,4,1,3,6,8,5,7)$ which is a solution to the classical dinner table problem but fails our condition for $k=4$.

\begin{figure}[h]
{\includegraphics[height=3.5cm]{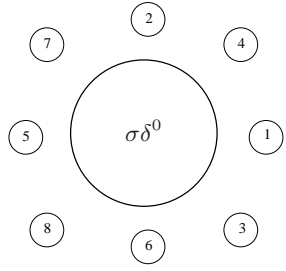}} 
\put(-60,47){\small $\sigma \delta^0$}
\caption{An arrangement  of 8 guests which is not convergent: the  guests $1,2,3,4$ (and $5,6,7,8$) are consecutive for both $\delta^0$ and $\sigma \delta^0$.}
\end{figure}

Now, given a  permutation $\sigma$ we associate a rational function and regular $n$-form
$$\widetilde{f}_{\sigma} = \prod_{i \in \Z/n\Z} { (z_i- z_{i+1}) \over (z_{\sigma(i)}-z_{\sigma(i+1)})} \qquad \hbox{ and  } \qquad 
\widetilde{\omega}_{\sigma} = { dz_1 \ldots dz_{n} \over \prod_{i \in \Z/n\Z} (z_{\sigma(i)}-z_{\sigma(i+1)})}
 $$ 
 on the space $\mathcal{C}^n=\{(z_1,\ldots, z_n) \in (\Pro^1)^n: z_i \neq z_j\}$ of configurations of $n$ distinct points in $\Pro^1$. 
They are defined up to an overall sign, which plays almost no role and shall be ignored.  They are both  $\mathrm{PGL}_2$-invariant. The former descends  to a rational function  $f_{\sigma}$ on  $\Mod_{0,n}  =\mathrm{PGL}_2 \backslash  \mathcal{C}^n$,
the latter, after  dividing by an invariant volume form 
on $\mathrm{PGL}_2$,  descends to a  regular $\ell$-form 
$\omega_{\sigma}$ on $\Mod_{0,n}$. 

Define the \emph{basic cellular integral} to be
\begin{equation} \label{IntroBasic}
I_{\sigma}(N) = \int_{\Dom_n} f_{\sigma}^N \omega_{\sigma}
\end{equation}
It converges if and only if $\sigma$ is a convergent permutation, as defined above.
 This integral can be written in the form $(\ref{IntM0n})$ by substituting  $(0, t_1, \ldots, t_{\ell}, 1,\infty)$ for  $(z_1,\ldots, z_n)$ 
 and formally omitting $dz_1 dz_{n-1} dz_n$ and all factors equal to $\infty$. 
 
 Clearly, the definition of the forms $\widetilde{f}_{\sigma}, \widetilde{\omega}_{\sigma}$, and hence $f_{\sigma}, \omega_{\sigma}$ only
 depend on the dihedral ordering defined by the permutation $\sigma$. Furthermore,  the  domain of integration in the integrals $(\ref{IntroBasic})$ admits a second dihedral symmetry of order $2n$, and these two dihedral symmetry groups define an equivalence relation on the set of permutations $\sigma$. Two equivalent permutations 
 give rise to the same family of integrals, and we call the equivalence class a \emph{configuration}. A list of convergent configurations for small $n$ is given in Appendix 2, together with the corresponding basic cellular integrals. As mentioned above, the basic cellular integrals reproduce
 Ap\'ery's theorems for $\zeta(2)$ and $\zeta(3)$ and a one-dimensional subfamily of the (two-parameter family) of  linear forms in odd zeta values used in the proof of Ball and Rivoal.

 A generalisation of this construction involves replacing $\widetilde{f}_{\sigma}^n$ with 
 $$\widetilde{f}_{\sigma}(\av,\bv) = \prod_{i \in \Z/N\Z} { (z_i- z_{i+1})^{a_{i,i+1} }\over (z_{\sigma(i)}-z_{\sigma(i+1)})^{b_{\sigma(i), \sigma(i+1)}}}  $$ 
 where $a_{i,i+1}$ and $b_{\sigma(i), \sigma(i+1)}$ are integers satisfying 
 $$   a_{\sigma_i -1 , \sigma_i } + a_{\sigma_i, \sigma_i + 1} = b_{\sigma_{i-1}, \sigma_i} + b_{\sigma_{i}, \sigma_{i+1}} $$
 for all indices $i$ modulo $\Z/n\Z$.  It descends again to a rational function $f_{\sigma}(\av,\bv)$ on $\Mod_{0,n}$ and we define an $n$-parameter family of integrals
 \begin{equation} \label{IntroGeneral}
I_{\sigma}(\av,\bv) = \int_{\Dom_n} f_{\sigma}(\av,\bv) \omega_{\sigma}
\end{equation}
 It converges under some  linear conditions on the indices $\av, \bv$,  and specialises to the basic cellular integrals if one sets all parameters
 equal to $n$. For $n=5,6$, this family of integrals reproduces precisely Rhin and Viola's integrals for $\zeta(2)$ and $\zeta(3)$, and
 hence gives the best irrationality measures for  these numbers which are presently known.\footnote{The record for $\zeta(2)$ has
 recently been broken by Zudilin \cite{Zu3}.} For other convergent configurations, it gives $n$-parameter generalisations of the Ball-Rivoal linear forms, and many new families which remain to be explored.

  \subsection{Contents} Section $\ref{sectMon}$ consists of reminders on moduli spaces $\Mod_{0,n}$ and basic facts about their geometry.
  Section $\ref{sectConfigs}$ defines  dinner table configurations, and establishes convergence properties for the corresponding cellular integrals.
  In section $\ref{sectPF}$ we show that the basic cellular integrals satisfy recurrence relations and study the effect of duality upon them.
  Section $\ref{sectGen}$ is concerned with properties of general cellular integrals, and section $\ref{sectMultStruct}$ 
studies a certain    multiplicative structure on cellular integrals coming from functorial maps between moduli spaces. 
  In section $\ref{sectOddZ}$, it is shown that a very specific family of configurations, after an appropriate change of variables, gives back 
  the linear forms in odd zeta values discovered by Ball and Rivoal. Section $\ref{sectCohom}$ discusses the vanishing problem $(4)$ from the cohomological point of view and  some more subtle   structures, such as Poincar\'e-Verdier duality, which are not obviously apparent from an inspection of integrals. 
  
  For the convenience of the reader, appendix $1$ gives a list of existing integrals from the literature which have led to the main diophantine results
  for zeta values. They nearly all
   arise as special cases of  generalised cellular integrals. Appendix $2$ tabulates some examples of basic cellular integrals in low degrees. Finally, appendix $3$ is devoted to a somewhat technical computation 
  of the motives underlying Apery's proofs of the irrationality of $\zeta(2)$ and $\zeta(3)$.

 \subsection{Outlook and related work}
 Whether the ideas in  this paper lead to new diophantine applications remains to be seen.  Insofar
 as it contains the linear forms  of \cite{B-R}, \cite{RV1}, \cite{RV2} as special cases,  it is fair to expect that it could lead to an improvement in quantitative
 diophantine results. The methods described  here also lead to new approximations to single odd zeta values such as $\zeta(5)$ (\S\ref{sectdualforms}) but it is unclear if they could lead to  an irrationality proof.   Much more optimistically still, one might  hope to prove the transcendence of $\zeta(3)$ by optimizing our polynomial forms in $\zeta(3)$ along the lines of \cite{Sorokin}.   Finally, it would be interesting to combine the geometric methods of this paper with the conditions on numerators studied  in  \cite{CFR, CFR2,Fi2, Zlob1} 
 to obtain  linear forms in antisymmetric multiple zeta values with odd arguments.
    There are connections with disparate subjects such as the theory of hypergeometric functions on the one hand,  and operads \cite{Alm} on the other, which remain to be explored.  It would also be interesting to  compare  our method with the quantum cohomology computations of \cite{Galkin, Golyshev}.
  \\
  
  \emph{ Acknowledgements}. This work was begun in 2004, and continued periodically during a stay at the Mittag-Leffler Institute in Stockholm in 2007,
  and finally written up at Humboldt University, Berlin and the Institute for Advanced Study, Princeton in 2014. I thank all three institutes for hospitality and financial support.
  The impetus to finish this project came from the recent implementations, due especially to  Erik Panzer, and independently,  Christian Bogner,
 for the algorithmic computation of Feynman integrals. Many thanks to Erik Panzer, Stephane Fischler, Clement Dupont  for comments and corrections, and Wadim Zudilin
 for many many helpful remarks.  This work was partially supported by ERC grant 257638 and NSF grant  DMS-1128155.
  
  \newpage
  
 \section{Moduli spaces $\Mod_{0,n}$: geometry and periods} \label{sectMon}

 \subsection{Coordinates}
Let  $n\geq 3$, and  let  $S$ denote a set with $n$ elements. Let
$\mathfrak{M}_{0,S}$ denote the moduli space of Riemann spheres
with $n$  points labelled with elements of $S$. If
$(\mathbb{P}^1)_*^S$ denotes the space of  $n$-tuples of distinct
points $z_s\in \Pro^1$, for $s\in S$, then
$$\mathfrak{M}_{0,S} =\mbox{PGL}_2 \backslash (\mathbb{P}^1)_*^S\
,$$
 where $\mbox{PGL}_2$ is the group of automorphisms of $\Pro^1$  acting  diagonally by
M\"{o}bius transformations. Throughout this paper,  we shall  set  
 \begin{equation} \label{elldef} \ell= n-3  \ .
 \end{equation} 
When $S=\{s_1,\ldots,
s_n\}$,  we often  write $i$ instead of $s_i$,
and $\Mod_{0,n}$ instead of $\Mod_{0,S}$. Since the
action of $\PGL_2$ on $\Pro^1$ is triply transitive, we
can place the coordinates $z_1$ at 0, $z_{n-1}$ at $1$, and $z_n$
at $\infty$ (note that this convention differs slightly from \cite{BrENS}). 

We define \emph{simplicial coordinates} $t_1,\ldots,
t_\ell$ on $\Mod_{0,n}$ to be:
\begin{equation} 
t_1=z_2\ ,\ \ldots \ ,  \ t_\ell=z_{\ell+1}\ \ .
\end{equation} 
The above argument shows that $\Mod_{0,n}$ is isomorphic to the complement of a  hyperplane arrangement
in affine space $\A^{\ell}$ of dimension $\ell$:
\begin{equation} \label{ExpSimpCoords}
\Mod_{0,n} \cong \{(t_1,\ldots, t_\ell) \in \A^\ell: \,
t_i\notin \{0,1\},\quad  t_i\neq t_j \hbox{ for all } i \neq j\}\
.\end{equation}
\emph{Cubical coordinates} $x_1,\ldots, x_{\ell}$ are defined by 
\begin{equation}\label{cubechangeofvars}
t_1=x_1\ldots x_\ell\ , \  t_2=x_2\ldots x_\ell\ , \  \ldots \ ,\
t_\ell =x_\ell\ . \end{equation} We can also identify
$\Mod_{0,n}$ with a complement of hyperbolae:
\begin{equation} \label{ExpCubeCoords}
\Mod_{0,n} \cong \{(x_1,\ldots, x_\ell) \in \A^\ell: \,
x_i\notin \{0,1\},\quad  x_i\ldots x_j\neq 1 \hbox{ for all } i <
j\}\ .\end{equation}

\subsection{Compactification} \label{sectCompact} There is a smooth projective compactification 
$\Mod_{0,S} \subset \overline{\Mod}_{0,S}$
 defined  by
Deligne, Mumford and Knudsen \cite{Knud} such that the complement
$\overline{\Mod}_{0,S} \backslash \Mod_{0,S}$ is a simple normal crossing
divisor.  A \emph{boundary divisor} $D$ is a union of irreducible components of $\overline{\Mod}_{0,S} \backslash \Mod_{0,S}$.  We  shall only require the following basic facts.

\begin{enumerate}
\item The irreducible boundary divisors are in one-to-one correspondence with \emph{stable partitions} $S= S_1 \cup S_2$, where $|S_1|, |S_2|\geq 2$ and $S_1 \cap S_2 = \emptyset$. They can be denoted by
$D_{S_1|S_2}$
 or simply $D_{S_1}$  (or $D_{S_2}$) when $S$ is clear from the context.
\item There is a canonical  isomorphism
$$D_{S_1 |S_2} \cong \overline{\Mod}_{0,S_1 \cup x} \times \overline{\Mod}_{0,S_2 \cup x}$$
which can be pictured as   a bouquet of two spheres joined at a point $x$,
with the points  $S_1$ lying in one of these spheres,  the points $S_2$ on the other.
\item Given two distinct stable partitions $S_1|S_2$ and $T_1|T_2$ of $S$, 
the divisors $D_{S_1|S_2}$  and $D_{T_1|T_2}$ have non-empty  intersection  if and only if
$$S_i \subseteq T_k \hbox{ and } T_{l} \subseteq S_j \qquad \hbox{ for some  } \{i,j\}=\{k,l\}= \{1,2\}\ .$$
\end{enumerate}
By taking repeated intersections of
boundary divisors one obtains a stratification on
$\overline{\Mod}_{0,S}$ by closed subschemes. The  irreducible strata of codimension $k$  are indexed by 
trees with $k$ internal edges and $|S|$ leaves labelled by every element of $S$. The large stratum $\overline{\Mod}_{0,S}$ is indexed by a corolla
with no internal edges, and divisors $D_{S_1|S_2}$ by two corollas with leaves labelled by  $S_1$ and $S_2$  respectively, joined along a single internal edge.
 The inclusion of strata corresponds to contracting internal edges on trees.

\subsection{Real components and dihedral structures} \label{sectRealCompandDihed}
The set of real points $\Mod_{0,S}(\R)$ is isomorphic to the space of configurations of 
$|S|$ distinct points on $\Pro^1(\R)$. Therefore the set of connected components is in one-to-one correspondence
$$\pi_0( \Mod_{0,S}(\R)) \quad \leftrightarrow \quad  \{\hbox{Dihedral structures on } S \}\ .$$
A dihedral structure   $\delta$ on $S$ is an equivalence class of cyclic orderings on $S$,
where a cyclic ordering is equivalent to its reversed ordering. If $S=\{1,\ldots, n\}$, we denote the  standard dihedral ordering $1<2<\ldots<n<1$ by
$\delta^0$.  In simplicial coordinates, the corresponding connected component is the open 
simplex
$$  \Dom_{\delta^0} = \{(t_1,\ldots, t_{\ell})\in \R^{\ell}\ : \   0 < t_1 < \ldots <  t_{\ell} <  1\}\ .$$ 
We say that an irreducible boundary divisor $D$ of  $\Mod_{0,S}$ is  \emph{at finite distance} with respect to  a dihedral
structure $\delta$ if $D(\C)$ meets the closure of $\Dom_{\delta^0}$ in $\overline{\Mod}_{0,S}(\C)$ in the analytic topology. 
Equivalently,  $D=D_{S_1|S_2}$ is at finite distance if and only if the elements of $S_1$ and $S_2$ are consecutive with respect to $\delta$.
Let $\delta_f$ denote the set of irreducible divisors at finite distance with respect to $\delta$.
If one depicts a dihedral structure $\delta$ as a set of points $S$ around a circle (up to reversing its orientation), then the set of divisors at finite distance correspond to chords in the circle which separate $S$ into the two subsets $S_1$ and $S_2$.
All remaining irreducible boundary divisors are said to be \emph{at infinite distance} with respect to $\delta$. 
Let $\delta_\infty$ denote the set of irreducible divisors at infinite distance with respect to $\delta$.

\begin{example} On $\Mod_{0,4}$,  $\delta^0_f= \{ D_{\{1,2\}| \{3,4\}} \ , 
 D_{\{2,3\} |  \{1,4\}} \} $  and $\delta^0_{\infty} = \{ D_{\{1,3\}| \{2,4\}} \}$. 
\end{example}

\subsection{Periods}

We shall consider periods of $\Mod_{0,S}$ of the form
\begin{equation} \label{GeneralI} 
I = \int_{\Dom_{\delta}} \omega
\end{equation} 
where $\omega \in \Omega^{\ell}(\Mod_{0,S};\Q)$ is a global regular $\ell$-form on $\Mod_{0,S}$. Such an integral can be written
as a $\Q$-linear combination of integrals $(\ref{IntM0n}).$ It converges if and only if the order of vanishing  along
all divisors $D$ at finite distance with respect to $\delta$ is non-negative:
\begin{equation}\label{convergenceasvD}
v_{D}(\omega)\geq 0 \qquad \hbox{ for all } D\in \delta_f . 
\end{equation}
The  \emph{singular locus} of $\omega$ is  defined to be the set of irreducible  divisors (necessarily boundary divisors) along which $\omega$ has a pole:
$$\mathrm{Sing}(\omega) = \{ D \hbox{ irred. s.t. }  v_{D}(\omega)<0\}$$
Then condition $(\ref{convergenceasvD})$ is equivalent to $\mathrm{Sing}(\omega) \subseteq \delta_{\infty}$.

The set of  permutations $\sigma$ of $S$ which preserve a dihedral structure $\delta$ is isomorphic to the dihedral group $D_{\delta}$ on $2|S|$ elements.
We have
\begin{equation}\label{dihedsymm}
  \int_{\Dom_{\delta}} \omega =  \int_{\Dom_{\delta}} \sigma^*(\omega) \qquad \hbox{ for all } \sigma \in D_{\delta}\ .
  \end{equation}

\begin{rem} It is sometimes convenient to write period integrals on $\Mod_{0,S}$ in terms of the dihedral `coordinates' $u_{c}$ indexed by chords in an $|S|$-gon, which were defined in \cite{BrENS}.
The  convergence condition $(\ref{convergenceasvD})$ and symmetry  $(\ref{dihedsymm})$ are obvious in these coordinates.
 \end{rem} 
 
Let $S=\{1,\ldots, n \}$.  In cubical coordinates, the domain $\Dom_{\delta^0}$ is isomorphic to the unit hypercube $[0,1]^{\ell}$, and 
a general period integral $(\ref{GeneralI})$ can be written in the form
\begin{equation}\label{GeneralCubicalInt}
\int_{[0,1]^{\ell}}  { P(x_1,\ldots, x_{\ell})  \over \prod_{1\leq i < j \leq \ell} (1 - x_i\ldots x_j)^{c_{ij}} }dx_1\ldots dx_{\ell}
\end{equation}
 where $c_{ij} \in \Z$ and $P$ is a polynomial with rational coefficients. All examples in Appendix $1$ are either of this form,
 or equivalent to it by a change of variables. 
 
 \subsection{Denominators} The algorithm of \cite{BrENS} for computing integrals $(\ref{IntM0n})$ by taking primitives
 in a bar complex is effective, and should lead to bounds on the denominators. The basic observation
 is that a differential form $P(x) dx$ where $P(x) \in \Z[x]$ is a polynomial of degree $n-1$, has a primitive 
 $$ \int P(x) dx  \in {1 \over d_n} \Z[x]$$
 where $d_n = \mathrm{lcm}(1,2,\ldots, n)$. The denominator is thus bounded by $d_n$, where $n$ is the order of the pole at infinity. An analysis
 of the steps in  \cite{BrENS}, working with $\Z$ coefficients, should lead to effective bounds on the denominators of the coefficients of the linear forms $\S\ref{sectStruct}$ in terms
 of the orders of the poles at infinity of the integrand.

 \section{Configurations and cellular integrals} \label{sectConfigs}
 
 \subsection{Convergent configurations}

 \begin{defn}
 A \emph{configuration} on a finite set  $S$ is an equivalence class $[\delta,\delta']$ of  pairs $(\delta, \delta')$ of dihedral structures on $S$
 modulo the equivalence relations 
 \begin{equation} \label{deltasequiv}
 (\delta, \delta') \sim (\sigma \delta, \sigma \delta')  \quad  \hbox{for  } \sigma \in \Sigma(S)\ .
 \end{equation} 
A pair of dihedral structures $(\delta, \delta')$   is \emph{convergent} if it  satisfies  (see \S\ref{sectRealCompandDihed}) \begin{equation}  \label{deltasdisjoint}
  \delta_f \cap \delta'_f = \emptyset\  .
  \end{equation}
  A configuration is convergent if it has  a convergent  representative $(\delta, \delta')$. 
  \end{defn}
 
 Let $\mathcal{C}_S$ denote the set of convergent configurations on $S$. We can view convergent configurations as pairs
 of connected components of $\Mod_{0,S}$ up to automorphisms
 \begin{equation} \label{embedCS}
 \mathcal{C}_S \quad \hookrightarrow   \quad \Sigma(S)  \backslash \big(\pi_0(\Mod_{0,S}(\R)) \times \pi_0(\Mod_{0,S}(\R)) \big) 
 \end{equation} 
 
  \begin{defn} The \emph{dual} of a  pair of dihedral structures  $(\delta, \delta')$ is 
  $$(\delta, \delta')^{\vee} =  (\delta', \delta)\ . $$ 
  It is well-defined on configurations, and  defines an involution $\vee: \mathcal{C}_S \rightarrow \mathcal{C}_S$.
 \end{defn}
 
In order to write down convergent configurations, it is convenient to identify $S$ with  $ \{1,\ldots, n\}$. 
A   pair of dihedral structures $(\delta, \delta')$ is equivalent to  $(\delta^0, \sigma \delta^0)$ where $\delta^0$ is the standard
dihedral ordering, where  $\sigma \in \Sigma(n)$ is a permutation  on $n$ letters. Define an equivalence relation  $\sigma \sim \sigma'$  on permutations if
$(\delta^0, \sigma \delta^0) \sim (\delta^0, \sigma' \delta^0)$, and denote the equivalence classes by $[\sigma]$.
The condition $(\ref{deltasdisjoint})$ is equivalent to the condition that no set 
of $k$ consecutive elements (where the indices are taken modulo $n$) 
$$\{\sigma_i, \sigma_{i+1}, \ldots, \sigma_{i+k}\}$$
is itself  a set of  consecutive integers modulo $n$, for all $2\leq k \leq n-2$. It does not depend on the choice of representative for $[\sigma]$.

The above  equivalence relation on permutations can be spelt out as follows.   Consider  the space of double cosets of bijections
 $$D_{2n} \backslash \mathrm{Bij(\{1,\ldots, n\}, \{1,\ldots, n\})}/D_{2n} $$
 where the  dihedral groups act  on the source and target respectively as symmetries of $\delta^0$.
A double coset is represented by  an equivalence class of permutations $(\sigma_1,\ldots, \sigma_n)$, where $\sigma \in \Sigma(n)$,  modulo the 
  group generated by    cyclic rotations
  \begin{eqnarray}
(\sigma_1,\ldots, \sigma_n)  &\sim & (\sigma_2,\ldots, \sigma_n,\sigma_1)  \nonumber   \\
  (\sigma_1,\ldots, \sigma_n) &\sim & (\sigma_1+1,\ldots, \sigma_n+1)   \nonumber
  \end{eqnarray} 
  where the entries are taken modulo $n$ in the second line,   and the reflections 
    \begin{eqnarray}
  (\sigma_1,\ldots, \sigma_n)& \sim &  (\sigma_n,\ldots, \sigma_1)  \nonumber  \\
  (\sigma_1,\ldots, \sigma_n) & \sim& (n+1-\sigma_1,\ldots, n+1- \sigma_n) \nonumber
    \end{eqnarray} 
Given such a class of permutations $\sigma$,  the pair of dihedral structures  $(\delta^0, \sigma \delta^0)$ is  well-defined modulo
the relations $(\ref{deltasequiv})$. 
 This establishes a bijection between configurations and equivalence classes of 
permutations. The dual of the  configuration corresponding to  $[\sigma]$ is the configuration represented by the inverse
permutation $[\sigma]^{\vee} = [\sigma^{-1}]$.

 \subsection{Cellular forms}
To any dihedral structure $\delta$ we associate the connected component $\Dom_{\delta}$ of $\Mod_{0,S}(\R)$.  This will serve as a domain of integration.
We can also associate a regular $\ell$-form as follows.  Let
 \begin{eqnarray}
  \widetilde{\omega}_{\delta}  =   \pm  \prod_{i\in \Z/n\Z} {dz_i \over z_{\delta(i)} - z_{\delta(i+1)} } \  \in \  \Omega^{\ell+3} ( (\Pro^1)^S_*;\Q) \ , \nonumber
 \end{eqnarray} 
 where the indices are taken modulo $n$.  
 Clearly  $ \widetilde{\omega}_{\delta}$ is  homogeneous of degree zero 
  and  is easily verified to be $\mathrm{PGL}_2$-invariant. Furthermore 
  $\sigma^* \widetilde{\omega}_{\delta} = \pm \widetilde{\omega}_{\sigma\delta}$ for any $\sigma \in \Sigma(S)$, which acts on $(\Pro^1)^S_*$
  by permuting the components. Let
  $$\pi:  (\Pro^1)^S_* \rightarrow \Mod_{0,S}$$
  be the natural map obtained by quotienting by   $\PGL_2$.  We  can divide  $\widetilde{\omega}$  by a  rational invariant volume form $v$  on $\PGL_2$ 
  to obtain  a differential form \cite{BCS}
  $$\omega_{\delta} \in \Omega^{\ell}(\mathcal{M}_{0,S};\Q)\ . $$
  It  satisfies $\pi^*(\omega_{\delta}) \wedge v = \widetilde{\omega}_{\delta}$ for any local trivialisation of $\pi$. If we  normalise  $v$  so that $\omega_{\delta}= \pm 1$, whenever $|S|=3$, then $\omega_{\delta}$ is unique up to a sign for all $\delta$.
 It follows that  \begin{equation} 
     \sigma^* \omega_{\delta} = \pm \omega_{\sigma(\delta)}    \quad   \hbox{ for all } \quad \sigma \in \Sigma(S)\ .   \end{equation} 
  
  \begin{rem} It is possible to fix all the signs by considering cyclic structures instead of dihedral structures, as was done in \cite{BCS}. The sign plays no role for us.
  \end{rem}

 \begin{lem}  \label{lempolesform}The form $\omega_{\delta}$ has a simple pole along every irreducible boundary divisor at finite  distance with respect to $\delta$, and no other poles.
 In other words, 
 $$\mathrm{Sing}(\omega_{\delta}) = \delta_f \ .$$
 \end{lem} 
 \begin{proof} This is  proposition 2.7 in \cite{BCS}.
 \end{proof} 
  When $S=\{1,\ldots, n\}$, the form $\omega_{\delta}$ can be written explicitly in simplicial coordinates  as follows. We can assume by dihedral symmetry that $\delta(n) = n$, in which case
   $$\omega_{\delta} =  \pm  { dt_1\ldots dt_{\ell} \over \prod_{i=1}^{n-2} ( t_{\delta(i)} - t_{\delta(i+1)} )}$$ 
where we write $t_0=0$ and $t_{n-1}=1$.

 \subsection{Basic cellular integrals}
 Given a pair of dihedral structures $(\delta, \delta')$  consider the 
  rational function  on $(\Pro^1)^*_S $ defined as follows:
 $$ \widetilde{f}_{\delta/ \delta'} =  \pm \prod_{i} {z_{\delta(i)} -z_{\delta(i+1)} \over z_{\delta'(i)} - z_{\delta'(i+1)} }   \ \in \ \Omega^0(   (\Pro^1)^*_S  ;\Q ) \ .$$
 It is  $\PGL_2$-invariant. It therefore descends to  a  rational function 
 $$ \pm f_{\delta/\delta'}  \ \in \  \Omega^0(   \Mod_{0,S}  ;\Q ) \ .$$
 Duality corresponds to inversion:
 \begin{equation}
  f_{\delta/ \delta'} =  \pm  \big(f_{\delta'/ \delta}\big)^{-1}  \ ,
 \end{equation} 
and furthermore,
\begin{equation} \label{fdrelatestwoomegas}
 f_{\delta/\delta'} \, \omega_{\delta} = \pm  \omega_{\delta'}\ .
\end{equation} 
 
  \begin{defn} For all $N\geq 0$, and any pair of dihedral structures $(\delta, \delta')$ define the family of \emph{basic cellular  integrals} to be
 \begin{equation} \label{INdefn}
 I_{\delta/\delta'}(N) = \Big| \int_{\Dom_{\delta}} \big( f_{\delta/\delta'} \big)^N \omega_{\delta'} \Big|
 \end{equation}
 It may or may not be finite.   \end{defn}
The numbers $I_{\delta/\delta'}(0)$ are essentially the cell-zeta values studied in \cite{BCS}.

\begin{lem} \label{lemconv} The integral $I_{\delta/\delta'}(N)$ is finite if and only if $(\delta, \delta')$ satisfy $\delta_f \cap \delta_f'=\emptyset$.
\end{lem}
\begin{proof} See \S\ref{sectProofconv}. \end{proof} 
Since $\Dom_{\delta}$, $\omega_{\delta'}$,  and $  f_{\delta/\delta'} $ are $\Sigma(S)$-equivariant (up to orientation and sign) we have
$$I_{\delta/\delta'}(N) = I_{\sigma \delta/ \sigma\delta'}(N) \quad \hbox{ for all } \sigma \in \Sigma(S)\ .$$
In particular, we obtain a well-defined map
\begin{eqnarray} 
 \mathcal{C}_S \times \N  & \To & \R^{\geq 0 }   \nonumber \\
 ([\delta, \delta'] , N) & \mapsto &I_{\delta/\delta'}(N)  \ .\nonumber
 \end{eqnarray}
 
 \begin{rem} Let $\delta$ be a dihedral structure with $\delta(n)=n$. Then the function $f_{\delta_0/\delta}$ can be written explicitly in simplicial coordinates
 $$f_{\delta_0/\delta} =  \pm {t_1 (t_2-t_1) \cdots (t_{\ell}-t_{\ell-1})(1- t_\ell)  \over \prod_{i=1}^{n-2} t_{\delta(i)-\delta(i+1)}} $$
 where $t_0=0$ and $t_{n-1}= 1$.
   In particular, it has no zeros or poles on the  open standard simplex $\Dom_{\ell}$, and is therefore either positive or negative definite on $\Dom_{\ell}$.
\end{rem} 

Notice that, in the case when $(\delta, \delta')$ is convergent,  we have
\begin{equation} \label{maxformula} 
I_{\delta/\delta'}(N) \leq \big( \max_{t\in \overline{\Dom}_{\ell}} | f_{\delta/\delta'}(t) | \big)^N I_{\delta/\delta'}(0)\ .
\end{equation}
One can write $f_{\delta/\delta'}$  as a product of dihedral coordinates \cite{BrCras}, \cite{BrENS}    which take  value in $[0,1]$ on $\overline{\Dom}_{\ell}$. This
 immediately implies  that the maximum  of $|f_{\delta/\delta'}(t)|$ on $\overline{\Dom}_{\ell}$ is strictly less than $1$. 
 One can certainly obtain much sharper bounds.
 
When $\sigma \in \Sigma(n)$ is a permutation,  we shall sometimes write 
 $$f_{\sigma}  \ \hbox{ for }\  f_{\delta^0/\sigma\delta^0} \quad \hbox{ and } \quad \omega_{\sigma}  \ \hbox{ for } \  \omega_{\sigma\delta^0} \ .$$
 
\subsection{Proof of convergence} \label{sectProofconv}
We wish to compute the order of vanishing of $f_{\delta/\delta'}$ along an irreducible boundary divisor $D$ of $\Mod_{0,S}$. 
For this, define
$$\I_D(i,j) = {1 \over 2}\big( \I(\{i,j\} \subset S_1 ) + \I(\{i,j\} \subset S_2)\big)$$
for any $i,j \in S$ and where $D$ corresponds to the stable partition $S_1\cup S_2$ of $S$. The symbol $\I$ on the right-hand
side denotes the indicator function: $\I(A\subset B)$ is $1$ if $A$ is a subset of $B$, and $0$ otherwise. In \cite{BrENS}, corollary 2.36,  it was shown that
$$\ord_D { (z_i-z_j)(z_k-z_l) \over (z_i-z_k)(z_j-z_l)} = \I_D(i,j) + \I_D(k,l) - \I_D(i,k)- \I_D(j,l)$$
where the cross-ratio on the left is viewed in $\Omega^0(\Mod_{0,S};\Q)$. If we define
$$\I_{D} (\sigma) = \sum_{i \in \Z/n\Z} \I_D(\{\sigma(i), \sigma(i+1)\})   \quad \in \quad  {1\over 2} \,\N\ ,$$
then it follows from the definition of $f_{\delta/\delta'}$ that
\begin{equation} \label{orderoff}
\ord_D f_{\delta/\delta'} = \I_D(\delta) - \I_D(\delta')\ .
\end{equation}

\begin{lem} \label{lemmaxn-2} Let $\sigma$ be a dihedral ordering on $S$. Then $\I_{D} (\sigma)\leq {n \over 2}-1$
with equality if and only if $D\in \sigma_f$, i.e., $D$ is at finite distance with respect to $\sigma$.
\end{lem} 
\begin{proof}
Let $D=D_{S_1 |S_2}$. Since $S_1 \subsetneq \{1, \ldots, n\}$, the  sum $ \sum_{i} \I(\{\sigma(i), \sigma(i+1)\}\subseteq S_1)$ is bounded above by $|S_1|-1$ and attains the maximum
if and only if the elements of $S_1$ are consecutive with respect to $\sigma$. Likewise for $ \sum_{i} \I(\{\sigma(i), \sigma(i+1)\}\subseteq S_2)$. Therefore
$\I_D(\sigma)\leq {1\over 2 } (|S_1|+|S_2|-2) = {n \over 2} -1$ with equality if and only if $D \in \sigma_f$.
\end{proof}

\begin{cor} \label{cor3.9} Let $\delta, \delta'$ be two dihedral structures. Then 
$$\ord_D f_{\delta/\delta'}> 0 \  \hbox{ for all  } \  D \in \delta_f \qquad  \Longleftrightarrow \qquad \delta_f \cap \delta'_f = \emptyset
$$
  \end{cor} 
 
 \begin{proof}  Let $D \in \delta_f$. By the previous lemma
  \begin{equation} \label{ordDineq}
  \ord_D(f_{\delta/\delta'}) = \I_{D}(\delta) - \I_{D}(\delta') = {n\over 2}-1 - \I_{D}(\delta')  
  \end{equation} 
 which is strictly positive if and only if $D \notin \delta'_f$.
 \end{proof}
We can now prove lemma $\ref{lemconv}$. Let $\delta, \delta'$ be two dihedral structures. Then by $(\ref{convergenceasvD})$, the integral $(\ref{INdefn})$ converges if and only if
\begin{equation}  \label{convineq} 
 \ord_D (f^N_{\delta/\delta'} \omega_{\delta'}) \geq 0 \qquad \hbox{ for all } D \in \delta_f\ . 
 \end{equation} 
Now, if $\delta_f \cap \delta'_f = \emptyset$, then lemma $\ref{lempolesform}$ implies that $\ord_D(\omega_{\delta'}) =0$, and the previous corollary
implies that  $\ord_D (f^N_{\delta/\delta'})\geq N$ for all $D \in \delta_f$. Therefore  $(\ref{convineq})$ holds.
On the other hand, if $D\in \delta_f \cap \delta'_f$ then $\ord_D(\omega_{\delta'}) =-1$ by lemma $\ref{lempolesform}$ and $\ord_D (f^N_{\delta/\delta'})\leq 0$
by $(\ref{ordDineq})$.  Therefore $(\ref{convineq})$ fails for this divisor $D$.

\begin{rem} The above argument shows that when the integral $I^{(\delta, \delta')}_N$  converges, its integrand vanishes to order at least $N$ along every boundary component of the compactification of the domain of integration:
$$  \ord_D (f^N_{\delta/\delta'} \omega_{\delta'})\geq N \qquad \hbox{for all } D \in \delta_f$$ 
 This  explains why it decays rapidly as $N\rightarrow \infty$.
\end{rem} 
The following lemma (which implies lemma \ref{lempolesform}) is stated here for later use.
\begin{lem} For any dihedral structure $\sigma$  and irreducible boundary divisor $D$
\begin{equation} \label{ordDomega}
\ord_D(\omega_{\sigma}) = {\ell-1\over2} -   \I_D(\sigma) \ .\end{equation} 
\end{lem} 
\begin{proof} See \cite{BCS} equation $(2.7)$.
\end{proof} 

\section{Picard-Fuchs equations}  \label{sectPF}
The family of basic cellular integrals satisfy interesting recurrence relations. We briefly sketch their properties.  Let $(\delta, \delta')$ be any convergent configuration and  fix signs of $f_{\delta/\delta'}, \omega_{\delta'}$ throughout this section. Define the generating series
$$F_{\delta/\delta'}(t) = \sum_{N\geq 0 } \int_{\Dom_{\delta}} f^N_{\delta/\delta'} t^N\,  \omega_{\delta'}  =  \int_{\Dom_{\delta}}  \omega_{\delta/\delta'}(t)$$
 where
$$\omega_{\delta/\delta'}(t) =  { 1 \over 1 - t f_{\delta/\delta'} }  \omega_{\delta'}\ . $$
The series $F_{\delta/\delta'}(t)$ converges for  $t \leq 1$ by $(\ref{maxformula})$  and the remarks which follow.   Define  a  hypersurface  in $\Mod_{0,S} \times \A^1$ by the vanishing locus of the following equation
$$H_{\delta/\delta'} \quad : \quad  1 - t f_{\delta/\delta'} = 0\ . $$
Let $\overline{H}_{\delta/\delta'}$ denote its Zariski closure in $ \overline{\Mod}_{0,S}\times \A^1$.
It is well-known how to construct, over some open $U \subset \A^1$ defined over $\Q$ containing $0$,    an algebraic vector bundle
$$ \mathcal{H}^{\rel}_{dR} = \mathcal{H}^{\ell}_{dR} (   ( \overline{\Mod}_{0,S} \backslash (\overline{H}_{\delta/\delta'} \cup A_{\delta'}), A_{\delta} \backslash  ((\overline{H}_{\delta/\delta'} \cup A_{\delta'} ) \cap A_{\delta}))/U 
 ) $$
 where $A_{\sigma}$ denotes the boundary divisor $\bigcup_{D\in\mathrm{Sing}(\omega_{\sigma})} D$ for any dihedral structure $\sigma$, 
 equipped with   an integrable Gauss-Manin connection 
$$\nabla^{\rel} : \mathcal{H}^{\rel}_{dR}\To \Omega^1_{U/\Q} \otimes_{\Or_{U/\Q}} \mathcal{H}^{\rel}_{dR}\  .$$
Its analytic vector bundle corresponds to a complex local system  $\mathcal{H}^{\rel}_B$. It is   the sheaf whose stalks at 
 $t\in U$ are   the relative singular cohomology groups
 \begin{equation}\label{relstalks}
( \mathcal{H}^{\rel}_B)_t= H_B^{\ell} ( \overline{\Mod}_{0,S} \backslash (\overline{H}_{\delta/\delta'}(t) \cup A_{\delta'}), A_{\delta} \backslash  ((\overline{H}_{\delta/\delta'}(t) \cup A_{\delta'} ) \cap A_{\delta}))\otimes \C
 \end{equation}
where   $\overline{H}_{\delta/\delta'}(t)$ denotes the fiber of $\overline{H}_{\delta/\delta'}$ over the point $t$.    Here, and later on in this section, $H^{\ell}_B(X,Y) $ denotes $H^\ell (X(\C),Y(\C))$, for $X,Y$  defined over $\Q$.

The differential form  $\omega_{\delta/\delta'}(t)$ is defined over $\Q$ and  has singularities in $\overline{H}_{\delta/\delta'} \cup A_{\delta'}$, so
defines  a   section   $\omega^{\rel}_{\delta/\delta'}$ of $\mathcal{H}^{\rel}_{dR}$.
There is a polynomial $ D_{\delta/\delta'}^{\rel}\in \Q[ t, {\partial_ t} ]$, 
such that 
$$ D_{\delta/\delta'}^{\rel}(\omega_{\delta/\delta'}^{\rel}(t) ) = 0$$
where $\partial_t$ stands for   $\nabla^{\rel}_{\partial/\partial t}$. Finally,
since the boundary of the closure of the real simplex $\Dom_{\delta}$  is contained in $A_{\delta}(\C)$, its 
 relative homology class  defines  a locally constant  section of   $(\mathcal{H}_B^{\rel})^{\vee}$ near $t=0$, and we obtain the  homogeneous Picard-Fuchs equation:
 $$ D_{\delta/\delta'}^{\rel} (F_{\delta/\delta'} (t) )= \int_{\Dom_{\delta}} D_{\delta/\delta'}^{\rel}  (\omega^{\rel}_{\delta/\delta'}(t) )  = 0  \ .$$
This equation is equivalent to a recurrence relation on the coefficients of $F_{\delta/\delta'}(t)$.

\subsubsection{Duality} The effect of duality  is not yet visible due to the asymmetric roles played by $\delta$ and $\delta'$.
To remedy this, consider, as above, the  algebraic vector bundle  (perhaps after making $U$ smaller) denoted by  
$$\mathcal{H}_{dR} =  \mathcal{H}^{\ell}_{dR} ( (\overline{\Mod}_{0,S} \backslash (\overline{H}_{\delta/\delta'} \cup A_{\delta'} \cup A_{\delta}))/U)$$
equipped with the Gauss-Manin connection $\nabla$. Its complex local system $\mathcal{H}_{B}$ has  stalks
$$(\mathcal{H}_{B})_t =  H_B^{\ell} ( \overline{\Mod}_{0,S} \backslash (\overline{H}_{\delta/\delta'}(t) \cup A_{\delta'} \cup A_{\delta})) \otimes \C$$
 at $t\in U$.
  As before, the class of  $\omega_{\delta/\delta'}(t)$   defines a section   $\omega_{\delta/\delta'}$  of $\mathcal{H}_{dR}$, which 
 is annihilated by  an operator we denote by $ D_{\delta/\delta'} \in  \Q[ t, {\partial_ t} ]$.

\begin{lem}  Let $(\delta, \delta')$ be convergent. For all $t\in \A^1$,  
$\overline{H}_{\delta/\delta'}(t) \cap A_{\delta}   = \emptyset$.
\end{lem} 
\begin{proof}  
By corollary \ref{cor3.9},  $ f_{\delta/\delta'}$ vanishes along  any irreducible component $D$  of $A_{\delta}$.  
 \end{proof}

There is a natural map $ (\mathcal{H}^{\rel}_{dR} , \nabla^{\rel}) \rightarrow (\mathcal{H}_{dR}, \nabla)$
which sends $\omega_{\delta/\delta'}^{\rel}$ to $\omega_{\delta/\delta'}$.
There is a corresponding map on local systems $\mathcal{H}_B^{\rel} \rightarrow \mathcal{H}_B$.  Let $t\in U$ and  consider the sets  
$$\begin{array}{lcl}
  U_1 = \overline{\Mod}_{0,S} \backslash A_{\delta}  & \supset    &   A_1 = \emptyset \\
  U_2 = \overline{\Mod}_{0,S} \backslash (\overline{H}_{\delta/\delta'}(t) \cup A_{\delta'})  & \supset  & A_2 = A_{\delta} \backslash  (A_{\delta'}  \cap A_{\delta})  \\
 \end{array}
$$ in $\overline{\Mod}_{0,S}$. The second $\supset$ follows from the previous lemma, which also gives
\begin{eqnarray}
 U_1 \cup U_2  & = &  \overline{\Mod}_{0,S} \backslash (A_{\delta}\cap A_{\delta'}) \qquad\qquad  \qquad  \supset \qquad A_1 \cup A_2 = A_2 \nonumber \\ 
 U_1 \cap U_2  & = &   \overline{\Mod}_{0,S} \backslash (\overline{H}_{\delta/\delta'}(t) \cup A_{\delta} \cup A_{\delta'}) \quad \,\,  \supset  \qquad A_1 \cap A_2 = \emptyset \nonumber
\end{eqnarray}
A relative Mayer-Vietoris sequence gives 
$$\rightarrow H_B^{\ell}(U_1 \cup U_2, A_1 \cup A_2) \rightarrow H_B^{\ell}(U_1,A_1) \oplus H_B^{\ell}(U_2,A_2) \rightarrow H_B^{\ell}(U_1\cap U_2, A_1 \cap A_2) \rightarrow $$
The natural morphism    $   \mathcal{H}^{\rel}_{B}   \rightarrow  \mathcal{H}_{B}$ corresponds  on stalks to the second map, restricted to the second component in the middle. Its  kernel has stalks  at $t\in U$ the image  of 
$$\ker \big(H_B^{\ell} (U_1 \cup U_2, A_1 \cup A_2) \rightarrow H_B^{\ell}(U_1,A_1)\big) $$
in $H^{\ell}_B(U_2,A_2)$.  This  has no dependence on  $t$, and is therefore constant.
 Thus the kernel of the  morphism $(\mathcal{H}^{\rel,an}_{dR}, \nabla^{\rel}) \rightarrow  (\mathcal{H}^{an}_{dR}, \nabla)$ of vector bundles on $U^{an}$ has  the trivial connection. It follows that  $D_{\delta/\delta'} \omega^{\rel}_{\delta/\delta'}$ is a section of  $\Or_{U}^m$ for some $m$.
By clearing denominators, we obtain  an  inhomogenous Picard-Fuchs equation of the form
\begin{equation} \label{InhomPF}
D_{\delta/\delta'}  F_{\delta/\delta'} (t)  = P_{\delta/\delta'}(t) 
\end{equation} 
where $P_{\delta/\delta'}(t) \in \C[t]$.  The duality is finally visible for the operator $D_{\delta/\delta'} $ because  on the open set $ \overline{\Mod}_{0,S} \backslash (\overline{H}_{\delta/\delta'}(t) \cup A_{\delta'} \cup A_{\delta})$, we have the identity
$$  t \omega_{\delta/\delta'}(t) = \pm  \omega_{\delta'/\delta}(t^{-1})   $$
which follows from equation $(\ref{fdrelatestwoomegas})$, and relates $D_{\delta/\delta'}$ and $D_{\delta'/\delta}$.
If we write $(\ref{InhomPF})$ as a recurrence relation between the coefficients of $F_{\delta/\delta'}$ in the form
\begin{equation} \label{prec} p_0(n) u_n + \ldots + p_{k}(n) u_{n+k}=0   
\end{equation}
where $p_i \in \Q[t]$, then the corresponding recurrence relation for  $F_{\delta'/\delta}$ is its dual:
\begin{equation} \label{dualrec}
p^{\vee}_0(n) u_n + \ldots + p^{\vee}_{k}(n) u_{n+k} =0
\end{equation}
where (after possibly multiplying $p^{\vee}_i$ by $(-1)^i$ owing to sign ambiguities), 
$$p^{\vee}_i(t) = p_{k-i}(-k-1-t) \qquad \hbox{ for all } 0\leq i\leq k\ .$$
In particular, if a convergent configuration $[\delta,\delta']$ is self-dual, then the coefficients of the generating series $F_{\delta/\delta'}(t)$   satisfy a recurrence relation  $(\ref{prec})$ whose  coefficients satisfy $p_i(t)= \lambda \,p_{k-i}(-k-1-t)$ 
for all $0\leq i \leq k$, and  some $\lambda \in\Q^{\times}$.

\begin{rem} The multiplicative structures on  cellular integrals \S\ref{sectMultStruct} implies that for certain
convergent configurations $[\delta_1, \delta_1']$ and $[\delta_2, \delta_2']$,  there exists a convergent
configuration $[\alpha, \alpha']$ such that  $F_{\alpha/\alpha'}(t)$ is the  Hadamard product 
of $F_{\delta_1/\delta_1'}(t)$ and $F_{\delta_2/\delta_2'}(t)$.
\end{rem}

\section{Generalised cellular integrals}\label{sectGen}
\subsection{Definition} Let $n=|S|\geq 5$ and let  $\delta,\delta'$ be a pair of dihedral structures on $S$.  Define  a rational function 
on $(\Pro^1)_S^*$ by the following formula:
 \begin{equation}
   \widetilde{f}_{\delta/\delta'} (\av,\bv)  =  \pm \prod_{i\in \Z/n\Z} {(z_{\delta_i} -z_{\delta_{i+1}})^{a_{\delta_i, \delta_{i+1}}}  \over (z_{\delta'_i} - z_{\delta'_{i+1}})^{b_{\delta'_{i},\delta'_{i+1}}}}   \end{equation} 
 where the indices $i$ are taken cyclically in $\Z/n\Z$ and  $\av=(a_{\delta_i,\delta_{i+1}})$, $\bv=(b_{\delta'_i,\delta'_{i+1}})$ are  integers satisfying the 
 homogeneity equations :
\begin{equation} \label{homequations}
a_{\delta_{i-1} , \delta_i } + a_{\delta_i, \delta_{i+1}} = b_{\delta'_{j-1}, \delta'_j} + b_{\delta'_{j}, \delta'_{j+1}}  \quad \hbox{ whenever  } \delta_i = \delta'_j\ ,
\end{equation} 
and  all indices are considered modulo $n$. With these conditions, $ \widetilde{f}_{\delta/\delta'} (\av,\bv)$ is $\PGL_2$-invariant and descends to a 
rational function $f_{\delta/\delta'}(\av,\bv) \in \Omega^0(\Mod_{0,S};\Q).$

\begin{defn} Let $S, \delta, \delta'$ be as above and suppose that $\delta_f \cap \delta'_f=\emptyset$. With  $\av, \bv$ parameters satisfying  $(\ref{homequations})$, define a generalised cellular form to be \begin{equation} \label{fsigmaabdef}
   f_{\delta/\delta'}(\av,\bv) \omega_{\delta'} \quad \in \quad \Omega^{\ell} (\Mod_{0,S};\Q)
\end{equation} 
Call the set of parameters $\av, \bv$  convergent if $(\ref{fsigmaabdef})$ has no poles along divisors $D$ at finite distance with respect to $\delta$. In this case define the generalised cellular integral to be
\begin{equation}\label{Igen} I_{\delta/\delta'} (\av,\bv) =  \pm  \int_{\Dom_{\delta}} f_{\delta/\delta'}(\av,\bv) \omega_{\delta'}\ .
\end{equation}
It converges by $(\ref{convergenceasvD})$. The action of $\Sigma(S)$ extends to an action on pairs of dihedral structures $(\delta, \delta')$ and also on parameters $\av, \bv$ 
by permuting indices. Clearly $(\ref{Igen})$   is invariant under this action, up to a sign.
\end{defn}
For any convergent pair $(\delta,\delta)'$, setting all $a_{i,j} , b_{k,l} =N$ clearly defines a solution to 
$(\ref{homequations})$ and gives back the basic cellular integrals $(\ref{INdefn})$. We now show that 
there is a non-trivial $n$-parameter family of convergent integrals of the form $(\ref{Igen}).$

\subsection{Parametrization}
Let $S=\{1,\ldots, n\}$ and let $\sigma\in \Sigma(n)$ be a choice of permutation  such that  $(\delta ,\delta') \sim (\delta^0,  \sigma \delta^0)$. 
To simplify the notations, 
we can write
 \begin{equation} \label{ftildesigmacase}
   \widetilde{f}_{\sigma} (\av,\bv)  =   \pm \prod_{i} {(z_{i} -z_{i+1})^{a_{i, i+1}}  \over (z_{\sigma_i} - z_{\sigma_{i+1}})^{b_{\sigma_{i},\sigma_{i+1}}}}   \end{equation} 
It descends to a rational function $f_{\sigma}(\av,\bv) \in \Omega^0(\Mod_{0,S};\Q)$.
In the case when $n$ is even, taking the alternating sum of the  equations $(\ref{homequations})$ yields
the  condition:
\begin{equation} \label{aequation} 
\sum_{i =1}^n (-1)^i (a_{\sigma_i -1 , \sigma_i } + a_{\sigma_i, \sigma_i + 1} )= 0
\end{equation} 
When $n=2k+1$ odd, the equations $(\ref{homequations})$ uniquely determine the  $b_{\sigma_{i}, \sigma_{i+1}}$
in terms of the $a_{i,i+1}$, and we  take the $a_{i,i+1}$ as parameters in $\Z^n$. This defines a map
\begin{eqnarray} \rho_{\sigma} \ : \  \Z^{2k+1}  & \To & \Omega^{\ell} (\Mod_{0,S};\Q) \label{evenmap} \\
 \av & \mapsto &  f_{\sigma }(\av,\bv) \, \omega_{\sigma} \nonumber
\end{eqnarray}
where $\bv$ is determined from $\av$ via $(\ref{homequations})$. In the case $n=2k$ is even, 
we can choose  a parameter    $b\in \bv$. Equation   $(\ref{aequation})$ defines a lattice $H_{\sigma} \subset \Z^{2k}$ isomorphic to $\Z^{2k-1}$. 
We can parametrize the space of generalised cellular integrands in this case by 
\begin{eqnarray} \rho_{\sigma,b} \ : \ H_{\sigma}\times \Z  & \To &    \Omega^{\ell} (\Mod_{0,S};\Q) \label{oddmap} \\
(\av, b) & \mapsto &     f_{\sigma}(\av,\bv)\,  \omega_{\sigma} \nonumber
     \end{eqnarray}
     where the remaining indices $\bv$ are determined   from $(\av,b)$ using $(\ref{homequations})$.

\begin{prop}  \label{prop4.2}
Let $D \subset \overline{\Mod}_{0,n}$ be an irreducible boundary divisor isomorphic to $\overline{\Mod}_{0,n-1}$ which is neither at finite distance
with respect to $\sigma\delta^0$ nor to $\delta^0$. Then
\begin{equation} \label{gencellvanishing}
v_{D} \big( f_{\sigma}(\av, \bv)  \,\omega_{\sigma} \big) \geq 0 
\end{equation}
for all $\av, \bv$ satisfying the equations $(\ref{homequations})$.  

Let $F \subset  \overline{\Mod}_{0,n}$ be an irreducible boundary divisor in $\delta^0_f$. Then
 \begin{equation} \label{OFD}
 v_F \big( f_{\sigma}(\av, \bv)  \,\omega_{\sigma} \big)  \geq   \sum_{i\in I } a_{i,i-1} - \sum_{j\in J} b_{\sigma_j, \sigma_{j+1}}
 \end{equation} 
where $|I| = n-2 $ and $|J| < n-2$ are certain subsets of $\{1,\ldots, n\}$ depending on $F$. It has  strictly fewer terms with negative coefficients
than with positive coefficients.
\end{prop}

\begin{proof}
By lemma $\ref{lempolesform}$,   $\omega_{\sigma}$ has no poles along any such divisor $D$. Therefore
$$v_{D}  \big( f_{\sigma}(\av, \bv)  \,\omega_{\sigma} \big) \geq  v_D ( f_{\sigma}(\av, \bv))\ .$$
For the latter, we have by $(\ref{ftildesigmacase})$
\begin{equation} 
v_{D}  ( f_{\sigma}(\av, \bv))= \sum_{i \in \Z/n\Z}   a_{i, i+1}  \I_D(\{i,i+1\}) -   b_{\sigma_i, \sigma_{i+1}}   \I_D(\{ \sigma_i,\sigma_{i+1}\}) \ .   \label{valformab} 
\end{equation}
A  divisor $D$ isomorphic to $\overline{\Mod}_{0,n-1}$ corresponds to a partition $S \cup T$ of $\{1,\ldots, n\}$,  where  $S=\{p,q\}$, and $p,q$ are not consecutive with respect to $\delta^0$ and $\sigma \delta^0$. Then
\begin{equation} \label{2vdfs} 
2 \, v_{D}  ( f_{\sigma}(\av, \bv))= \sum_{i \in \Z/n\Z}   a_{i, i+1} \, \I(\{i,i+1\} \subseteq T)  -  b_{\sigma_i, \sigma_{i+1}}  \, \I(\{ \sigma_i,\sigma_{i+1}\} \subseteq T) 
\end{equation} 
where $T$ is the complement of $S$ in $\{1,\ldots, n\}$. If we denote by 
$$A = \sum_{i\in \Z/n\Z}   a_{i, i+1}  \qquad \hbox{ and }  \qquad B = \sum_{i\in \Z/n\Z}   b_{\sigma_{i}, \sigma_{i+1}}$$
then equation $(\ref{homequations})$ implies that $A=B$. Adding $B-A$ to $(\ref{2vdfs})$ gives
$$2\, v_{D}  ( f_{\sigma}(\av, \bv))= a_{p-1,p} + a_{p,p+1} + a_{q-1,q} + a_{q,q+1}
  -  b_{p_{-},  p }- b_{p , p_+ } -b_{q_{-},  q }-b_{q,  q_+ }   
$$
where $p_{\pm}$ are the adjacent neighbours of $p$ with respect to the ordering $\sigma$, and likewise for $q$. 
This quantity vanishes  by  $(\ref{homequations})$ and  proves the first part.

For the second part, consider a stable partition $S, T$ where the elements of $S$ and $T$ are consecutive (with respect to the standard dihedral ordering).
Then  by lemma $\ref{lemmaxn-2}$, the first sum in equation  $(\ref{valformab})$ yields  $n-2$  terms  which occur with a plus sign, and the second sum contributes  at
most $n-3$ terms, which occur with a  minus sign. 
\end{proof}

\begin{rem}  \label{remWeakCell} The conditions $(\ref{gencellvanishing})$ mean  that the integrand is `weakly cellular' in the sense that its polar locus is   contained
in the  set of divisors $\sigma_f$  with certain extra  divisors corresponding to stable partitions $S\cup T$ where $|S|, |T| \geq 3$.
With a little more work, one can find  further  constraints on the set of   extra divisors which can occur, and yet more constraints under the assumption that the integrand is convergent.
\end{rem}

The region of convergence for generalised forms in parameter space is defined by hyperplane inequalities. We know it is not compact because it contains
the infinite family of basic cellular integrals. The following corollary shows that it is genuinely $n$-dimensional (i.e. not contained in a hyperplane).

\begin{cor}  Let $R_{\sigma}$ denote the region in  the parameter space $(\ref{evenmap})$ or $(\ref{oddmap})$, depending on the parity of $n$, which  consists 
of points  corresponding to convergent forms.  Consider the region 
 $C^n \subset \N^n$   defined for all $n\geq 0 $ by 
$$C^n= \{ ( x_1,\ldots, x_n) \in \N^n :  \hbox{ for all }i, |x_i  - m| < \textstyle{m\over n^2}  \hbox{ for some } m \in \N\}\ . $$
It contains the diagonal $\N\hookrightarrow \N^n$. Then if $n=2k+1$ is odd, 
$$C^{2k+1} \  \subset \  R_{\sigma}  \ \subset \  \N^{2k+1}$$
and if $n=2k$ is even, 
$$C^{2k+1} \cap (H_{\sigma} \times \Z) \   \subset  \ R_{\sigma}  \ \subset  \ (H_{\sigma} \cap \N^{2k})\times \Z$$
\end{cor}
\begin{proof} For the upper bound, observe that an application of formula $(\ref{valformab})$ for the order of vanishing along a divisor at finite distance $D$ corresponding
to a stable partition $\{i,i+1\} \cup \{1,\ldots, i-1, i+2,\ldots, n\}$ gives 
\begin{equation} \label{valforeasydivisors}
\ord_D  \big( f_{\sigma}(\av, \bv)  \big) = a_{i,i+1}
\end{equation}
One verifies using  $(\ref{ordDomega})$ that  $\ord_D  (\omega_{\sigma}) =0$ for such divisors $D$.
Therefore, by $(\ref{convergenceasvD})$, convergence requires that all indices $a_{i,i+1}$ be $\geq 0$. 

For the lower bound,  consider the case $n=2k$. The case when $n$ is odd is similar. 

 Choose a cyclic ordering on $\sigma$, and assume without loss of generality that $b=b_{\sigma_n,\sigma_1}$.  Consider
 a linear form $(\ref{valformab})$. 
 Substitute  an equation $(\ref{homequations})$ of the form
 \begin{equation} \label{brelas} 
 b_{\sigma_1,\sigma_2} = a_{\sigma_2 -1, \sigma_2} + a_{\sigma_2, \sigma_2+1} - b_{\sigma_2, \sigma_3}
 \end{equation} 
to replace  the indeterminate $b_{\sigma_1,\sigma_2}$ with its successor 
 $b_{\sigma_2,\sigma_3}$. Proceed in this manner until  $(\ref{valformab})$ is written in terms of the $a_{i,i+1}$ and $b$ only.
At the end  there will be  at most $n^2$ terms. 
Furthermore, the sum of the positive coefficients will exceed the sum of the negative coefficients by at least one by proposition $\ref{prop4.2}$
and since $(\ref{brelas})$ preserves the number of positive minus the number of negative terms.
 Any  linear form with these properties
takes positive values on $C^n$.  By $(\ref{OFD})$, this region is contained in $R_{\sigma}$.
\end{proof}
The above upper and lower bounds  on $R_{\sigma}$ can easily be improved if one wishes.

\subsection{Examples}
\subsubsection{Dixon's integrals \cite{Dixon} for $1,\zeta(2)$}
Let $n=5$ and  $\sigma = (5,2,4,1,3)$.  
 The generalised cellular integrand is
 $$\widetilde{f}_{\sigma}(\av,\bv)= \pm { (z_1-z_2)^{a_{1,2}}(z_2-z_3)^{a_{2,3}} (z_3-z_4)^{a_{3,4}} (z_4-z_5)^{a_{4,5}} (z_5-z_1)^{a_{5,1}} \over  (z_5-z_2)^{b_{5,2}} (z_2-z_4)^{b_{2,4}}(z_4-z_1)^{b_{4,1}}(z_1-z_3)^{b_{1,3}}(z_3-z_5)^{b_{3,5}}    } $$
 where the exponents satisfy $a_{1,2}+ a_{2,3} = b_{5,2}+b_{2,4}$, $\ldots$, $a_{5,1}+ a_{1,2} = b_{4,1}+b_{1,3}$.
 Since $n$ is odd, we can take as our set of parameters $a_i= a_{i,i+1}$ for $i \in \Z/5\Z$ and solve for the $b_{i,j}$.
The generalised cellular  integral in simplicial coordinates is
$$\int_{0\leq t_1 \leq t_2 \leq 1} { t_1^{a_1} (t_2-t_1)^{a_2} (1-t_2)^{a_3}  \over   t_2^{b_{1,3}} (1-t_1)^{b_{2,4} } } {dt_1 dt_2 \over (1-t_1)t_2} $$
where $b_{1,3} = a_1+a_2-a_4$ and $b_{2,4}= a_2+a_3  -a_5$, which follows from  solving the homogeneity equations $(\ref{homequations})$.
By $(\ref{valforeasydivisors})$ and $(\ref{ordDomega})$, the valuation of the integrand along the divisor 
$D_{\{12\}|\{345\}}$ is $a_{1,2}$.  There are exactly five divisors at finite distance obtained from this one by cyclic symmetry, and therefore the convergence conditions are exactly 
$a_{i} \geq 0$  for  $i \in \textstyle{ \Z/ 5\Z}$.
Now one can change variables to transform the previous integral into cubical coordinates $t_1 =xy, t_2=y$. This results in the integrals
$$I(a_{1}, a_{2}, a_{3}, a_{4}, a_{5}) = \int_{[0,1]^2} { x^{a_{1}} (1-x)^{a_{2}} y^{a_{4}} (1-y)^{a_{3}}   \over  (1-xy)^{a_{2}+a_{3}-a_{5}} } {dx dy \over 1-xy} $$
which coincide with $(\ref{AIRV2})$.  This family of integrals has a large group of symmetries  \cite{RV2}. 
A geometric derivation of these transformations  in terms
of natural morphisms between moduli spaces  was given in \cite{BrENS} \S7.7.

This family of integrals yields linear forms in $1$ and $\zeta(2)$.  The order of  vanishing
of the integrand along  the  five divisors at infinity   are  
$a_{i-2}-a_i-a_{i+1}-1$ for $i \in \Z/5\Z$.  If any of these forms is $\geq 0$,  the coefficient of $\zeta(2)$ in $I$ vanishes, by 
 lemma \ref{lemcompletebroken}.

 \subsubsection{ Rhin-Viola's integrals for $\zeta(3)$}
Let $n=6$ and $\sigma = (1,4,2,6,3,5)$.  We choose parameters $a_i= a_{i,i+1}$ and $b= b_{3,6}$. Then equation $(\ref{aequation})$ is the equation
$$a_{4}+ a_{5} = a_{1}+a_{2} \ .$$
The generalised cellular form (up to an overall sign  chosen to ensure that it is positive on the simplex $0\leq t_1\leq t_2\leq t_3 \leq 1$)
is 
\begin{equation} \label{gencellz3} {   t_1^{a_{1}}(t_2-t_1)^{a_{2}}(t_3-t_2)^{a_{3}}(1-t_3)^{a_{4}}
\over t_3^{b_{1, 4}}(t_3-t_1)^{b_{2, 4}}(1-t_2)^{b_{3, 5}}  } {dt_1 dt_2 dt_3 \over t_3(t_3-t_1)(1-t_2)} \ .
\end{equation} 
Using the homogeneity equations $(\ref{homequations})$ it can be rewritten in terms of our  parameters via $b_{1,4}= a_{6}+a_{3}-b$, $b_{2,4}= a_{4}-a_{ 6}+b$,  and 
$b_{3,5}= a_{2}+a_{3}-b$.
The convergence conditions  for the six divisors obtained from $\{12|3456\}$
by cyclic permutations lead, by equation $(\ref{valforeasydivisors})$ and $(\ref{ordDomega})$, to inequalities
$ a_{i } \geq 0$  for all  $i\in \Z/6\Z\ $.
There are three further divisors at finite distance, which, on applying $(\ref{valformab})$ yield the following linear forms, which can 
be reduced to our choice of parameters using $(\ref{homequations})$:
\begin{eqnarray}
123|456  &  : & \textstyle{1\over 2} (a_{1, 2}+a_{2, 3} +a_{4, 5} +a_{1,6}+ 2) \ = \  a_{1}+a_{2} +1         \ \geq 0  \nonumber  \\
126|345   &  :  & \textstyle{1\over 2} (a_{1, 2} + a_{3, 4}+a_{4, 5} +a_{1,6}   -b_{2, 6} -b_{3, 5}) \ = \   a_{4}   +b    -a_{2}  \ \geq 0  \nonumber \\
156|234   &  : &  \textstyle{1\over 2} a_{2, 3}+a_{3, 4}+a_{5, 6}+a_{1, 6}  \ = \  a_{6}+a_{2}+a_{3}-a_{4}-b    \ \geq 0 \nonumber 
\end{eqnarray}
Thus, in our choice of parameter space, the region of convergence is defined by the second and third  hyperplanes (since the first linear form is trivially non-negative).

\begin{lem} Rhin and Viola's family of integrals $(\ref{AIRV3})$ for $\zeta(3)$ coincides, up to reparametrization,  with the family  of generalised cellular integrals 
for $\sigma = (1,4,2,6,3,5)$. 
\end{lem} 
\begin{proof} Pass to cubical coordinates $t_1 = xyz, t_2=yz, t_3=z$, and rename the parameters
$ ( a_{1}, a_{2}, a_{3}, a_{4}, a_{6}, b) $  by $( l, s ,k ,q,r, r-q-h+s+k)$
respectively. 
Then $(\ref{gencellz3})$ leads to the family of period  integrals on $\Mod_{0,6}$ of the form
$$\int_{[0,1]^3}{x^l(1-x)^{s}y^{l+s} (1-y)^k z^{l+s-q} (1-z)^q \over (1-xy)^{k-h+s} (1-yz)^{h+q-r} } {dxdydz \over (1-xy) (1-yz)}\ ,$$
depending on the six new parameters $h,k,l,q,r,s$. The convergence conditions above translate into the inequalities 
$h,l,s,k,q,r\geq0$, $l+s-q \geq 0$, $r+k-h \geq 0.$ 
 This is exactly the family of  integrals $(\ref{AIRV3})$, after applying the change of variables
\begin{equation}\label{ChangeofVarRV3}
(x,y,z) \mapsto \Big( 1-xy, {1-y \over 1-xy}, z\Big)\ .
\end{equation}
\end{proof}

In particular, this family of integrals gives linear forms in $1,\zeta(3)$ by \cite{RV2}.

\subsubsection{Generalised cellular family for  $\sigma = (8,2,7,3,6,4,1,5)$} \label{sectMo8ex}
Choose as parameters $a_i= a_{i,i+1}$ for $i\in \Z/8\Z$ and $b=b_{5,8}$.
The equation $(\ref{aequation})$ is then the relation 
$$ H_{\sigma} : \qquad a_{6}+a_{7} + a_{8} = a_{2}+a_{3}+a_{4}$$
A  reduced set of convergence conditions are given by $a_{i} \geq 0$  for all $i \in \Z/8\Z$, and
$$
\begin{array}{lrl}
  128|34567&  a_{1}  + b -a_{7} &  \geq 0   \\
 1278|3456  &  a_{3}+a_{4}-a_{6}   &   \geq 0  \\
 1234|5678  &    a_{5}+a_{6}+a_{7}-b  +1 &    \geq 0  \\
 456|12378    & a_{1}+a_{2}+b  -a_{6}-a_{7} & \geq    0 \nonumber 
\end{array}
$$
The corresponding divisor is indicated on the left hand-side. All other convergence conditions are a  consequence of  these ones.
The integral is given by 
$$I(\av,b) = \int_{\Dom_5}  { t_1^{a_{1}} (t_2-t_1)^{a_{2}}(t_3-t_2)^{a_{3}}(t_4-t_3)^{a_{4}}(t_5-t_4)^{a_{5}}(1-t_5)^{a_{6}} \over 
 (1-t_1)^{b_{2, 7}} (1-t_2)^{b_{3, 7}} t_3^{b_{1, 4}} t_4^{b_{1, 5}}  (t_5-t_2)^{b_{3, 6}} (t_5-t_3)^{b_{4, 6}} } \omega_{\sigma}$$
 where 
$$ \omega_{\sigma} = {dt_1 \ldots dt_5  \over (1-t_1)(1-t_2)(t_5-t_2)(t_5-t_3)t_3t_4}$$
and the parameters in the denominator are given by
\begin{eqnarray}
b_{2,7}= a_{1}+a_{6}-a_{3}-a_{4}+b \nonumber & , &     b_{4,6} = a_{4}+a_{5} +a_{6}+a_{7}-a_{1}-a_{2} -b  \\
b_{3,7} = a_{3}+a_{4} + a_{7}-a_{1}-b \nonumber & , & b_{1,4} = a_{1}+a_{2} + a_{3}-a_{5} -a_{6}-a_{7}+b\\
b_{3,6}  = a_{1} + a_{2}-a_{4}-a_{7}+b \nonumber  &, & b_{1,5} =  a_{4}+a_{5} -b  
\end{eqnarray}
Using the symbolic integration programs due to Erik Panzer \cite{Panzer, PanzerProg}, or \cite{BBog1}, one can compute many examples of such generalised cellular integrals 
and finds experimentally that they are linear combinations of $1, \zeta(3), \zeta(5)$ only.\footnote{Wadim Zudilin has
very recently proved that this family of integrals is equivalent to another family  considered by Viola \cite{Viola}, 
and similar to integrals in  \cite{Zlob1}, \cite{Zlob2} (private communication)}.
I made a half-hearted attempt to search for $I(a_1,\ldots, a_7,b)$ in which the coefficient of $\zeta(3)$ vanishes. Tantalisingly, I found
the following examples, which could  be part of an infinite sequence of approximations to $\zeta(5)$ (perhaps after applying a symmetry argument or
modifying the numerators of the family $I(\av,b)$),
or could just be accidental:
\begin{eqnarray}
 I(  1, 0, 0, 1, 0, 0, 0, 0)  & =  & 2\zeta(5) - 2  \nonumber \\
I( 2,0,0,2,0,0,0,0 ) & = & 2 \zeta(5)  - {33\over 16}  \nonumber \\ 
I( 3, 2, 0, 3, 2, 0, 2, 2) & = &  60 \zeta(5) -  {161263 \over 2592} \nonumber 
\end{eqnarray}
In a different direction, a residue computation shows that a  large family of these integrals has vanishing $\zeta(5)$ coefficient, and hence gives linear forms in $1,\zeta(3)$.  It can be made explicit by applying a version of lemma \ref{lemcompletebroken} and computing
the order of vanishing of the integrand along a cellular boundary divisor corresponding to $\sigma$.  It would be interesting to know whether this 
leads to  new approximations to $\zeta(3)$.

\section{Multiplicative structures} \label{sectMultStruct}

There are  partial multiplication laws  between  cellular integrals generated by `product maps' between moduli spaces. 
\subsection{Product maps}
Let $S$ be a set with $n\geq 3$ elements, and let $S_1, S_2 \subset S$ be subsets satisfying
$$
|S_1 \cap S_2| = 3 \qquad \hbox{ and } \qquad S= S_1 \cup S_2\ .
$$
A \emph{product map}, defined in  \cite{BrENS} \S2.2, \S7.5, is the product of forgetful maps
$$m: \Mod_{0,S} \To \Mod_{0,S_1} \times \Mod_{0,S_2}\ .$$
It follows from the assumptions on $S_1$ and $S_2$ that  is an open immersion, and that the dimensions of the source and target are equal.

Now if $\delta_1,\delta_2$
are dihedral structures on $S_1,S_2$ then 
$$m^{-1} (\Dom_{\delta_1} \times \Dom_{\delta_2}) = \bigcup_{\delta: \delta|_{S_i} = \delta_i}  \Dom_{\delta}$$
where the union is over the set of dihedral structures $\delta$ on $S$ whose restrictions to $S_1, S_2$ coincide with $\delta_1,\delta_2$.  Let $\omega_i \in \Omega^{|S|_i-3}(\Mod_{0,S_i})$ for $i=1, 2$. Then
\begin{equation}  \label{genproduct}
  \int_{\Dom_{\delta_1}} \omega_1 \times \int_{\Dom_{\delta_2}} \omega_2 = \int_{m^{-1} (\Dom_{\delta_1} \times \Dom_{\delta_2})} m^* (\omega_1 \otimes \omega_2)
  \end{equation}
 in the case when all terms converge.
This formula can be used to multiply two cellular integrals. It  gives a third cellular integral under certain conditions  on $\delta_1,\delta_2$. 
\subsection{Multiplication of pairs of dihedral structures}
Fix a set  
$$T = \{1,2,3\}$$
 on three elements, with the ordering $1<2<3$.
 Define a \emph{triple} in a set $S$  with $n\geq 3$ elements to be an injective map 
$t:T \hookrightarrow S.$

\begin{defn}  \label{defnmultipliable} Let $(\delta, \delta')$ be a pair of dihedral structures on $S$, and let $t:T\hookrightarrow S$ be a triple.
  We say that $(\delta, \delta')$ is  \emph{multipliable along} $t$ if:

\begin{enumerate}
\item  The elements $t(1),t(2),t(3)$,  in that order,  are consecutive  with respect to  $\delta$.
\item  The elements $t(1), t(3)$ are consecutive with respect to  $\delta'$.
\end{enumerate}
We say that a configuration on $S$ is \emph{multipliable}, if for some, and hence any representative $(\delta, \delta')$,
there exists a triple $t$ in $S$ satisfying $(1)$ and $(2)$.
\end{defn}

\begin{rem} Note that  $(\delta, \delta')$ is multipliable along $t:T \hookrightarrow S$ if and only if it is multipliable along   
$\widetilde{t}:T \hookrightarrow S$, where $(\widetilde{t}(1), \widetilde{t}(2), \widetilde{t}(3) ) =(t(3),t(2),t(1)).$
\end{rem}

Suppose that we have  two pairs of dihedral structures  $(\delta_1,\delta_1')$ on $S_1$ and $(\delta_2, \delta_2')$  on $S_2$, and $t_i:T\hookrightarrow S_i$,
for $i=1,2$  such that 
$$(\delta_1, \delta_1') \quad \hbox{ and }  \quad (\delta_2, \delta_2')^{\vee}=(\delta_2', \delta_2) \quad  \hbox{ are multipliable along } t_1, t_2 \hbox{ respectively}\ . $$
Note that it is the \emph{dual} of $(\delta_2, \delta_2')$ which must be multipliable along $t_2$. Let
$$S= S_1 \cup_{t_1=t_2} S_2$$
denote the  disjoint union of the sets $S_1$ and $S_2$ modulo the  identification $t_1(i)=t_2(i)$ for $i=1,2,3$.
Finally, we can define the product to be
$$(\delta_1, \delta_1') \star_{t_1,t_2} (\delta_2,\delta_2') =  (\alpha, \alpha')$$
where $\alpha$ (respectively $\alpha'$) is the unique dihedral structure on $S$ whose restrictions to $S_i$
coincide with $\delta_i$  (respectively, $\delta'_i$) for $i=1,2$.  In the language of \cite{BCS},
$\alpha$ is a relative shuffle of (cyclic structures representing) $\delta_1$ and $\delta_2$, and similarly for $\alpha', \delta_1', \delta_2'$. 

\begin{example} \label{exampleprod} In the following examples, a tuple $(s_1,\ldots, s_n)$  denotes the dihedral structure in which $s_i$ are arranged consecutively around a circle (considered modulo reflections).
Firstly,  the pair of dihedral structures 
$$\big((p_1,p_2,\mathbf{p_3},\mathbf{p_4},\mathbf{p_5}) , (p_2,\mathbf{p_4},p_1,\mathbf{p_3},\mathbf{p_5})\big) $$ 
 is multipliable along $(1,2,3) \mapsto (p_3,p_4,p_5)$.
Consider the pair of dihedral structures 
$$\big((\mathbf{q_1},q_2,q_3,\mathbf{q_4},\mathbf{q_5},q_6) , (q_6,q_2,\mathbf{q_4},\mathbf{q_1},\mathbf{q_5},q_3)\big)\ .$$
Its dual is multipliable along $(1,2,3) \mapsto (q_4,q_1,q_5)$. Let $S= \{ p_1,p_2,p_3,p_4,p_5,q_2,q_3,q_6\}$.
 The product of these two dihedral structures is 
 $$ \big(  ( p_1,p_2,p_3,q_3,q_2,p_4,q_6,p_5), (p_3,p_1,p_4,p_2,p_5,q_3,q_6,q_2) \big) $$
The  configuration class of this pair is denoted by ${}_8 \pi_1$ in appendix $1$.
\end{example}

 \begin{figure}[h]
{\includegraphics[height=2cm]{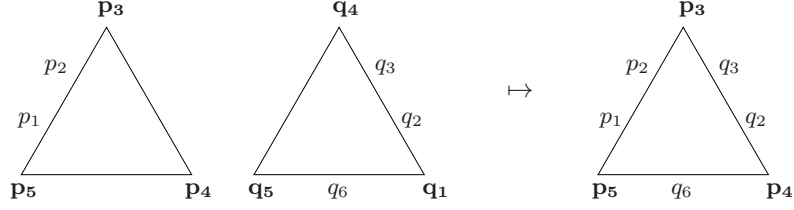}} 
\put(-255,62){\small $\mathbf{p_3}$}\put(-222,-7){\small $\mathbf{p_4}$}\put(-288,-7){\small $\mathbf{p_5}$}\put(-285,20){\small $p_1$}\put(-275,40){\small $p_2$}
\put(-166,62){\small $\mathbf{q_4}$}\put(-132,-7){\small $\mathbf{q_1}$}\put(-198,-7){\small $\mathbf{q_5}$}\put(-168,-7){\small $q_6$}  \put(-140,20){\small $q_2$} \put(-150,40){\small $q_3$}
\put(-35,62){\small $\mathbf{p_3}$}\put(-2,-7){\small $\mathbf{p_4}$}\put(-68,-7){\small $\mathbf{p_5}$}\put(-65,20){\small $p_1$}\put(-55,40){\small $p_2$}
\put(-38,-7){\small $q_6$}  \put(-10,20){\small $q_2$} \put(-20,40){\small $q_3$}
\put(-100,30){$\mapsto$} 
\caption{A graphical depiction of the left-hand factors in example $\ref{exampleprod}$: a dihedral structure can be depicted as a set of points around the triangle $t_i(T)$. The triangle on the right is obtained by superimposing  the two triangles on the left and identifying $p_3=q_4$, $p_5=q_5$, $p_4=q_1$.}
\end{figure}

Note that the multiplication laws on configurations are not unique:  two configurations can have  different representatives which 
multiply together in different ways.
\subsection{Multiplication of generalised cellular integrals}

\begin{lem}   \label{lemextend} Let $(\delta, \delta')$ be a pair of dihedral structures which is multipliable along $t: T \hookrightarrow S$, and let 
$a_{\delta_i,\delta_{i+1}}$, $b_{\delta'_{i},\delta'_{i+1}}$ denote parameters which are defined whenever 
$$\{\delta_i, \delta_{i+1} \} \not \subset t(T) \quad  (\hbox{resp. }   \{\delta'_i, \delta'_{i+1} \} \not \subset t(T) )$$
and satisfy the homogeneity equations
$$a_{\delta_{i-1} , \delta_i } + a_{\delta_i, \delta_{i+1}} = b_{\delta'_{j-1}, \delta'_j} + b_{\delta'_{j}, \delta'_{j+1}}  \hbox{ if   } \delta_i = \delta'_j\ ,
  $$
  whenever   all terms are defined. Then there is a unique way to define parameters $a_{\delta_i,\delta_{i+1}}$, $b_{\delta'_{i},\delta'_{i+1}}$  for all $i$ 
  such that  the homogeneity equations $(\ref{homequations})$ hold.

  Furthermore, if all parameters  $a_{ij}$ and $b_{ij}$ which were initially defined were equal to $N$, then the extended set of parameters are also all equal to $N$.
\end{lem} 
\begin{proof} For simplicity, we can renumber labels so  that $t(T)=(1,2,3)$, and the dihedral orderings 
can be assumed to be  of the form $$\delta=( \ldots, p_{1\ell}, 1,2,3, p_{3r}, \ldots) \qquad \hbox{ and } 
\delta' = ( 1, q_{1r},\ldots, q_{2\ell}, 2, q_{2r}, \ldots, q_{3\ell},3)$$
The extra homogeneity equations which need to be satisfied are of the form
\begin{eqnarray}
a_{p_{1\ell},1} + a_{1,2}  & =&  b_{3,1} + b_{1, q_{1r}}  \nonumber \\
a_{1,2} + a_{2,3}  & =&  b_{q_{2\ell}2} + b_{2 q_{2r}}  \nonumber \\
a_{2,3} + a_{3,p_{3r}}  & =&  b_{q_{3\ell} 3} + b_{3 1}  \nonumber 
\end{eqnarray}
which can be uniquely solved for $a_{1,2}, a_{2,3}$ and $b_{3,1}$. 
\end{proof}
The lemma is clearly  true  if one replaces a pair of dihedral structures by their dual.

Now let $(\delta_1,\delta'_1)$ and $(\delta_2, \delta'_2)^{\vee}$ be two pairs of dihedral structures on $S_1, S_2$  and let $(\alpha, \alpha')$ denote their product with respect to $t_i:T_i \hookrightarrow S_i$  for $i=1, 2$ as defined above. Then $\alpha,\alpha'$ are dihedral  structures on $S = S_1 \cup_{t_1=t_2} S_2$.

Consider a set of parameters $\av, \bv$: 
$$a_{\alpha_i, \alpha_{i+1}}\in \Z \quad \hbox{ and } \quad b_{\alpha'_i, \alpha'_{i+1}}\in \Z$$
which satisfy the homogeneity equations $(\ref{homequations})$. Since the restrictions of $\alpha, \alpha'$ to $S_i$ 
are $\delta_i, \delta_i'$,  the restriction of the sets of parameters $\av,\bv$ to those terms whose indices lie in  $S_i$ satisfy the conditions of lemma \ref{lemextend}. 
Denote their extensions defined  in  the lemma by $\av_i,\bv_i$, for $i=1,2$.

\begin{prop}  With the above notations,
\begin{equation} I_{\alpha/\alpha'}( \av, \bv ) =   \pm  I_{\delta_1/\delta_1'}( \av_1, \bv_1 ) \, I_{\delta_2/\delta_2'}( \av_2, \bv_2 )
\end{equation} 
The  integral  on the left-hand side is finite if and only if  both integrals on the right-hand side are finite.

In particular, in the case when all parameters are equal to $N$, we obtain a multiplicative formula for basic cellular integrals
\begin{equation} I_{\alpha/\alpha'}(N) =  I_{\delta_1/\delta_1'}(N) I_{\delta_2/\delta_2'}(N) \quad \hbox{ for all } N  \geq 0 \ .
\end{equation}
\end{prop} 

\begin{proof} The proof is an application of the product formula $(\ref{genproduct})$.
Let  $$m: \Mod_{0,S} \To \Mod_{0,S_1} \times \Mod_{0,S_2}$$ be the product map.
By construction, $\alpha$ is the unique dihedral  structure on $S$ which restricts to $\delta_1$ and $\delta_2$, so the domain of integration is  $m^{-1}( \Dom_{\delta_1} \times \Dom_{\delta_2}) = \Dom_{\alpha}$.

Therefore it suffices to show that
$$m^*(  f_{\delta_1 /\delta_1'}(\av_1,\bv_1) \omega_{\delta_1'}   \otimes   f_{\delta_2 /\delta_2'}(\av_2,\bv_2)  \omega_{\delta_2'})  =  \pm \, f_{\alpha /\alpha'}(\av,\bv)   
 \omega_{\alpha'}\ .$$
The case when all parameters are equal to zero is the identity
$m^*(\omega_{\delta'_1}\otimes \omega_{\delta'_2}) = \pm \omega_{\alpha'}$  
 and follows from \cite{BCS}, proposition 2.19, since $\alpha'$ is the unique dihedral structure which restricts to $\delta'_1, \delta_2'$. It remains to show that 
 \begin{equation}\label{mstarf} 
 m^*(  f_{\delta_1 /\delta_1'}(\av_1,\bv_1) \otimes   f_{\delta_2 /\delta_2'}(\av_2,\bv_2)  )  =  \pm \, f_{\alpha /\alpha'}(\av,\bv)   \ .
 \end{equation} 
 To see this, denote the marked points of $S_1$ by $x_1,\ldots, x_r$ 
 and $S_2$ by $y_1,\ldots, y_s$.  Using $\PGL_2$, we can place the elements of $t_1(T)\subset S$ at $0,1,\infty$, and so 
 $(x_1,x_2,x_3) = (0,1,\infty) = (y_1,y_2,y_3)$.   The expression
 $f_{\delta_1 /\delta_1'}(\av_1,\bv_1) \otimes   f_{\delta_2 /\delta_2'}(\av_2,\bv_2)  $  is given by  the limit, as both  $(x_1,x_2,x_3)$ and $(y_1,y_2,y_3)$ tend to $(0,1,\infty)$, of 
 $$  \prod_{i\in \Z/r\Z}   { (x_{\delta_1(i)} -x_{\delta_1(i+1)})^{a_{\delta_1(i), \delta_1(i+1)}} \over 
 (x_{\delta'_1(i)} -x_{\delta'_1(i+1)})^{b_{\delta'_1(i), \delta'_1(i+1)}} }  \times \prod_{i\in \Z/s\Z}   { (y_{\delta_2(i)} -y_{\delta_2(i+1)})^{a_{\delta_2(i), \delta_2(i+1)}} \over 
 (y_{\delta'_2(i)} -y_{\delta'_2(i+1)})^{b_{\delta'_2(i), \delta'_2(i+1)}} } $$
since  terms of the form  $(x_i-x_j)^n$ or $(y_i-y_j)^n$ where $i,j\in \{1,2,3\}$, will drop out in the limit.
Denote the marked points of $S$ by $(z_1,\ldots,  z_{r+s-3}). $ 
The map $m$ sends
$$(z_1,\ldots, z_{r+s}) \mapsto (z_1,\ldots, z_r) \times (z_1,z_2,z_3, z_{r+1},\ldots, z_{r+s})$$
By definition of the dihedral structures  $\alpha, \alpha'$  the left-hand side of $(\ref{mstarf})$  is exactly the limit as $(z_1,z_2,z_3)$
tends to $(0,1,\infty)$ of the following expression
 $$   \prod_{i\in \Z/(r+s-3)\Z}   { (z_{\alpha_i} -z_{\alpha_{i+1}})^{a_{\alpha_i, \alpha_{i+1}}} \over 
 (z_{\alpha'_{i}} -z_{\alpha'_{i+1}})^{b_{\alpha'_{i}, \alpha'_{i+1}}} }  $$
 which is  $f_{\alpha /\alpha'}(\av,\bv)$ up to a  sign. \end{proof}

 \section{Linear forms in odd zeta values} \label{sectOddZ}
 
The  Ball-Fischler-Rivoal integrals, which can be  used \cite{B-R, RiCRAS,  FiBourbaki} to prove that the $\Q$-vector space generated by odd zeta values is infinite
dimensional, are a  special case of  a certain family of generalised  cellular integrals. 

\begin{defn}Let $m\geq 3$. Consider the family of  convergent configurations defined by  the equivalence class
of permutations
\begin{equation}  \label{piodd} \pi_{odd}^{m}   =  (2m , 2, 2m-1, 3, 2m-2, 4, \ldots, m, 1, m+1)
\end{equation} 
\end{defn} 
  
\begin{prop} Let $1\leq r<m$. With the special choice of parameters 
$$ a_{m,m+1} = a_{2m,1} = b_{m+1,2m} = b_{m,1} = rn $$
and setting all other parameters $a, b$ equal to $n$, the generalised cellular integrals
$I(\av,\bv)$ coincide with the family of integrals $(\ref{AIRiv2})$ where  $a=2m$. 

In particular, the basic cellular integrals for $\pi^m_{odd}$ correspond to the case $r=1$.
\end{prop}
\begin{proof} 
Set $p=m-1$, and $\pi= \pi^m_{odd}$.
 The integrand corresponding to this configuration is
 $  f_{\pi}^n g^{rn} \omega_{\pi}$
 where $f_{\pi}$ is the basic cellular integrand (all parameters equal to $1$), and $g$ is the function 
  represented by the cross-ratio
$$g={ (z_m-z_{m+1})(z_{2m}-z_1) \over  (z_{m+1}-z_{2m})(z_m-z_1)     } \ .$$
Writing  these  in cubical coordinates $x_1,\ldots, x_{2p-1}$ gives
 $$f_{\pi} = { \prod_{i=1}^{p-1} (x_i\ldots x_{2p-i}) \prod_{i=1}^{2p-1} (1-x_i) \over 
 \prod_{i=1}^{p-1}    ( 1- x_i \ldots x_{2p-i} )(1- x_{i+1} \ldots x_{2p-i} )} \quad , \quad g = {x_p-1 \over x_p}\ . $$
 Perform the following   change of variables in two stages. First set  $x_{p}= 1-s_{2p-1}$, and  
$$
x_i = { s_{2i-1} - 1 \over s_{2i}-1} \quad  \ , \quad   x_{p+i} = { s_{2p-2i} - 1 \over s_{2p-2i+1}-1} \quad \hbox{ for }  \ 1\leq i <p \ .$$
Next perform the  change of variables $s_i = y_1\ldots y_i$ for $1\leq i \leq 2p-1$. One easily verifies that this  gives the integral $(\ref{AIRiv2})$. The details are somewhat tedious and are omitted, but 
 one  checks that  in the new variables is:
$$f_{\pi} = {\prod_{i=1}^{2p-1}  y_i(1-y_i)  \over (1-y_1\ldots y_{2p-1}) \prod_{i=1}^{p-1} (1-y_1\ldots y_{2i})}  \qquad , \qquad 
g=  {  y_1\ldots y_{2p-1}   \over 1- y_1 \ldots y_{2p-1}}$$
The differential form  $\omega_{\pi} $ can be computed similarly and becomes
$$ {   dy_1 \ldots dy_{2p-1}  \over   (1-y_1\ldots y_{2p-1}) \prod_{i=1}^{p-1} (1-y_1\ldots y_{2i})}  $$
Finally, one must check that the above change of variables defines a homeomorphism of the unit hypercube $\{(x_1,\ldots, x_{2p-1}): 0\leq x_i \leq1\} $
with $\{(y_1,\ldots, y_{2p-1}): 0\leq y_i \leq1\} $.  

The generalised cellular integrals for the above choice of parameters is then 
$$\int_{[0,1]^{2p-1}} f_{\pi}^n g^{rn} \, \omega_{\pi} $$
which coincides with the family of integrals $(\ref{AIRiv2})$.
\end{proof}
In particular, by \cite{RiCRAS, B-R} this family of integrals yield linear forms in odd zeta values
$1$, $\zeta(3)$,  \ldots ,  $\zeta(2m-3)$. 
It is highly likely that the same holds for $I_{\pi^m_{odd}}(\av,\bv)$ (including the example of \S\ref{sectMo8ex}), for any convergent  values of the parameters.

\begin{rem}  The evidence suggests that these families of integrals have   symmetry groups  and identities of hypergeometric type generalising
those discovered by Rhin and Viola.  It would be interesting to study these groups with a view to applying  the group method 
of Rhin and Viola to linear forms in odd zeta values. \end{rem}

\begin{prop}  Let $m\geq 2$. The generalised cellular integrals corresponding to  the sequence of convergent configurations
\begin{equation} 
\pi^{m}_{even} =   (2m+1 , 2, 2m, 3, 2m-1, 4, \ldots, m+2, 1, m+1)
\end{equation} 
 with parameters given by 
$$ a_{1,2m+1} = a_{m+1,m+2} = b_{m+1, 2m+1} = b_{1,m+2} = rn$$
and all other parameters equal to $n$,  are  equal to the family  $(\ref{AIRiv2})$ where $a=2m-1$.
\end{prop}
\begin{proof} Write the integrand corresponding to this configuration in cubical coordinates $x_1,\ldots, x_{2m-2}$ .
Perform the change of variables, $x_{m}= 1-s_{2m-2}$, and  
$$
x_i = { s_{2i-1} - 1 \over s_{2i}-1} \quad \hbox{ for }  \ 1\leq i <m   \ , \quad   x_{m+i} = { s_{2m-2i} - 1 \over s_{2m-2i+1}-1} \quad \hbox{ for }  \ 1< i \leq m \ .$$
A final change of variables $s_i = y'_1\ldots y'_i$ gives the integral $(\ref{AIRiv2})$.
\end{proof}
This family seems to  yield linear forms in even zeta values
$1$, $\zeta(2)$, \ldots, $\zeta(2m-2)$
 for all values of the parameters.
Note that there are many other families with an (apparently) similar property  such as the following family for all $n\geq 2$:
$$ (2n+1, n, 2n-1, n-1,  \ldots, 2, n+2, 1, n+1) $$
It would be interesting to know if they can be used to improve on the presently known transcendence measures for $\pi^2$.

\subsection{The dual linear forms} \label{sectdualforms}
The generalised cellular integrals of the configurations  $ \big(\pi_{odd}^{m} \big)^{\vee} $ which are dual to $(\ref{piodd})$ experimentally produce linear forms in 
$$1, \zeta(2), \ldots , \zeta(2m-6), \zeta_{2m-3}$$
where $\zeta_{2m-3}$ is a polynomial in odd zeta values and even powers of $\pi$ of weight $2m-3$.
As discussed in \S\ref{sectCohom}, we can define motivic versions of the generalised cellular integrals taking values in motivic multiple zeta values.
It now makes perfect sense to project the $\zetam(2)$ to zero, yielding linear forms in $1$ and $\zetam(2m-3)$ only. Taking the period gives linear forms 
in $1$ and $\zeta(2m-3)$. These linear forms are often  small.

\begin{example} Consider  the case  $m=4$, denoted ${}_8 \pi^{\vee}_8$ in Appendix 1. Then
$$\omega_{ {}_8 \pi^{\vee}_8} = {dt_1\ldots dt_5 \over   (t_{1}-t_{3})t_{3}(1-t_{4})(t_{4}-t_{2})(t_{2}-t_{5}) } $$
and an example of a generalised cellular integral is:
$$\int_{\Dom_{8}}{ t_{1}^8(t_{1}-t_{2})^8(t_{2}-t_{3})^8(t_{3}-t_{4})^7(t_{4}-t_{5})^8(t_{5}-1)^8 \over (t_{1}-t_{3})^6t_{3}^9(1-t_{4})^9(t_{4}-t_{2})^6(t_{2}-t_{5})^{10}} \omega_{ {}_8 \pi^{\vee}_8} = a_0 + a_1 \zeta(2) + a_2 \zeta_5$$
where $\zeta_5 = 2 \zeta(2) \zeta(3) + \zeta(5)$ and $a_0,a_1,a_2 \in \Q$. Either by computing with motivic multiple zeta values, or working with relative cohomology classes, one can ensure the coefficients $a_i$  are well-defined. We obtain using \cite{PanzerProg} a linear form 
$a_0 + a_2 \zeta(5) $ where
$$ a_0=  -{48144548550856003417243773593 \over 19289340000} \ , \ a_2 = { 2407028604043866880} $$
The  $\Z$-linear form obtained by clearing denominators is less than $1$, which is what is required for an irrationality proof.  There are many similar examples. An infinite family of such examples
would suffice to prove the irrationality of $\zeta(5)$.
\end{example}

 \section{Cohomology} \label{sectCohom}
  
 A proper understanding of  problem $(4)$  seems to require cohomological and motivic methods.
For this reason, I include a brief  discussion  of these ideas.
 
 \subsection{Moduli space motives}
The integrals $(\ref{IntM0n})$ are  periods of the motives considered in \cite{GM}.
For $|S|\geq4$   let $A,B \subset \overline{\Mod}_{0,S}$
be a pair of  boundary divisors such that $A$ and $B$ have no common irreducible components.  Let $\ell = |S|-3$ and define
$$m(A,B) = H^{\ell}( \overline{\Mod}_{0,S} \backslash A, B \backslash (B \cap A))$$
in the category $\MT(\Z)$ of mixed Tate motives over $\Z$. In particular, it has a de Rham realisation $m(A,B)_{dR}$ which is a finite dimensional graded vector space over $\Q$, and a Betti realisation $m(A,B)_B$ which is a finite dimensional vector space over $\Q$, equipped with an increasing weight filtration $W$. 
There is a comparison isomorphism
$$ \mathrm{comp}_{B, dR} : m(A,B)_{dR} \otimes_{\Q} \C \overset{\sim}{\To} m(A,B)_B \otimes_{\Q}\C$$
which is compatible with weight filtrations, where the weight filtration on $m(A,B)_{dR}$ is the filtration associated to its grading.
A convergent period integral of the form 
$$I = \int_{\Dom_{\delta}} \omega  \qquad \hbox{ where }\quad \omega \in \Omega^{\ell} (\overline{\Mod}_{0,S} \backslash A;\Q)  $$
    can be interpreted as follows. 
Let 
$ A = \cup_{D \in \mathrm{Sing} (\omega)} D$ and  $B= \cup_{D \in \delta_f} D$.
By $(\ref{convergenceasvD})$,  $A$ and $B$ have no common irreducible components.
The integrand $\omega$ defines a relative cohomology class $[\omega] \in  m(A,B)_{dR}$ via the surjective map of global forms
$$  \Omega^{\ell}( \overline{\Mod}_{0,S} \backslash A, B \backslash (A \cap B);\Q) \To \Omega^{\ell}( \overline{\Mod}_{0,S}\backslash A;\Q)$$
It is surjective because  the irreducible components of $B$ have dimension $\ell-1$ and so the restriction of $\omega$ to $B$ necessarily vanishes.
On the other hand, the domain $\Dom_{\delta}$ defines a relative homology cycle in singular (Betti) homology of the underlying complex manifolds with $\Q$ coefficients:
$$[\Dom_{\delta}]  \in H^B_{\ell} (\overline{\Mod}_{0,S} \backslash A, B \backslash (B \cap A))   =  \big( H_B^{\ell} (\overline{\Mod}_{0,S} \backslash A, B \backslash (B \cap A))\big)^{\vee}$$
Thus we have $[\omega] \in m_{dR}(A,B)$ and $[\Dom_{\delta}] \in m(A,B)^{\vee}_B$, and the period integral can be interpreted via the Betti-de Rham comparison map
$$ \int_{\Dom_{\delta}} \omega = \langle \mathrm{comp}_{B, dR} \, [\omega] , [\Dom_{\delta}] \rangle \in \C$$
The  pair of divisors $A, B$ - which are described by combinatorial data - determine the numbers which can occur in the 
previous integral, as we shall presently explain.

\subsection{Motivic periods and vanishing} 
We refer to \cite{BrSVMP}, \S2 for background on motivic periods.  The ring of motivic periods of $\MT(\Z)$ is defined to be
 $$P^{\mm} = \Or( \mathrm{Isom}_{\MT(\Z)}(\omega_{dR}, \omega_B))\ .$$
 It is a graded ring, equipped with a  period homomorphism
$$\mathrm{per} : P^{\mm} \To \C$$
by evaluating on $ \mathrm{comp}_{dR, B}$. 
We apply this construction to integrals on moduli spaces.  Let $\omega$, $\Dom_{\delta}$, be as above and  define the motivic period integral to be 
$$I^{\mm} (\omega, \Dom_{\delta})= [ m(A,B) , [\omega], [\Dom_{\delta}]]^{\mm}  \quad \in \quad P^{\mm} \ , $$ 
which is  the function $\phi \mapsto \langle \phi(\omega), \Dom_{\delta} \rangle: \mathrm{Isom}_{\MT(\Z)}(\omega_{dR}, \omega_B) \rightarrow \A^1$.
Its period is 
 $$ \mathrm{per}\, I^{\mm}(\omega, \Dom_{\delta}) =   \int_{\Dom_{\delta}} \omega \ . $$
\begin{thm} The motivic period  $I^{\mm}(\omega,\Dom_{\delta})$ is a $\Q$-linear combination of motivic multiple zeta values of weights $\leq \ell$.
Furthermore, if 
$$\gr^W_{2m} m(A,B)=0$$
then the coefficients of  motivic multiple zeta values of weight $m$ in $I^{\mm}(\omega,\Dom_{\delta})$  vanish.
\end{thm} 

\begin{proof} The motivic period $ I^{\mm}(\omega, \Dom_{\delta})$ is in fact a real, effective motivic period because
$\Dom_{\delta}$ is invariant under real Frobenius, and $m(A,B)$ has weights in $[0,\ell]$ (see \cite{BrSVMP}, \S2). The first part follows from \cite{BrSVMP}, proposition  7.1 (i), which is a corollary of \cite{BrMTZ}. \footnote{This proof uses the main theorem of \cite{BrMTZ} and is not effective. It would be interesting
to have a version along the lines of proof of \cite{BrENS} which actually enables one to control denominators.}

 Now  let $\{[\omega^{(m)}_i]\}$ be a basis for $\gr^W_{2m} \, m(A,B)_{dR}$ for $0\leq m\leq \ell$. Then
there exist rational numbers  $a^{(m)}_i \in \Q$ such that 
$$[ \omega] = \sum_{i,m}  a_i^{(m)}  \,   [\omega^{(m)}_i] $$
and hence, by bilinearity of motivic periods, 
$$I^{\mm} (\omega, \Dom_{\delta})= \sum_{i,m}  a_i^{(m)}  \, [m(A,B) , [\omega^{(m)}_i], [\Dom_{\delta}]]^{\mm} \ , $$
where $ [   m(A,B) , [\omega_i^{(m)}], [\Dom_{\delta}]]^{\mm} \in P^{\mm}$ are motivic periods of weight $m$, since the weight-grading
is determined from the de Rham grading \cite{BrSVMP} $(2.13)$. The second part is immediate.
\end{proof}
Applying the period homomorphism immediately gives the
\begin{cor} The integral $I$ is a $\Q$-linear combination of multiple zeta values of weights $\leq \ell$. 
If  $\gr^W_{2m} m(A,B)_{dR}$ vanishes, then  this linear combination does not involve multiple zeta values of weight $m$.
\end{cor}
Thus a simple-minded method to achieve vanishing is to find boundary divisors $A,B$, such that  certain graded pieces
of the de Rham cohomology $m(A,B)_{dR}$ vanish.  This is possible for Ap\'ery's approximations to $\zeta(2)$ and $\zeta(3)$ (Appendix 3).

\begin{rem} A more promising approach to force vanishing of coefficients, which I have not explored,   is via  representation theory. Suppose that there is a finite group $G$ which acts upon 
$m(A,B)_{dR}$ (for instance, via  birational transformations of $\Mod_{0,S}$). Then each graded piece
$(m(A,B)_{dR})_n$
is a finite-dimensional $\Q[G]$-module. Let $V$ be an irreducible representation of $G$  over $\Q$ and $\pi_V$ the corresponding projector. Consider the motivic periods 
$ \pi_V\, I^{\mm}(\omega, \Dom_{\delta}) = I^{\mm}(\pi_V\omega, \Dom_{\delta})\ . $
If the representation $V$ does not occur in a component $ \gr^W_{2m} m(A,B)_{dR}$, then  $\pi_V\,I^{\mm}(
\omega, X)$ cannot 
contain a motivic  multiple zeta value of weight $m$. 
 \end{rem}
 \subsection{Remarks on the Galois coaction}  The ring of motivic periods carries an action of the de Rham  motivic Galois group
 $G^{dR}= \mathrm{Isom}(\omega_{dR}, \omega_{dR})$.   This is equivalent to a coaction by $\Or(U^{dR})$, 
 where $U^{dR}$ is the unipotent radical of $G^{dR}$. General nonsense provides an abstract formula for this coaction (see for example \cite{BrSVMP},  equation $(2.12)$.)
 
 \begin{problem} Find a combinatorial formula for the motivic    coaction on the $I^{\mm}(\omega, \Dom_{\delta})$. 
  \end{problem}
  
The analogous problem for motivic multiple zeta values is known, due to Goncharov, Ihara, and  \cite{BrMTZ}. 
The reason this problem is relevant for irrationality questions  is the fact  that a motivic period which is primitive for this coaction
is necessarily a linear combination of \emph{single} motivic zeta values only (\cite{BrMTZ}). Thus a solution to this problem would give a  criterion for obtaining linear forms in single zeta values, as opposed to multiple zeta values.
  It is already an interesting problem  to try to prove geometrically that the examples of \S\ref{sectOddZ} are primitive.

\subsection{Duality} Poincar\'e-Verdier duality states that 
\begin{equation}\label{mABdual}
m(A,B)= m(B,A)^{\vee} \otimes \Q (-\ell)
\end{equation}
where both sides have weights in the interval $[0,2\ell]$, since $\overline{\Mod}_{0,S}$ is smooth projective and $A\cup B$ normal crossing. In particular,
\begin{equation}
\gr^W_m m(A,B)_{dR} \cong \gr^W_{2\ell-m} m(B,A)_{dR} \ ,
\end{equation}
which enables us to transfer vanishing theorems from $m(A,B)$ to $m(B,A)$.

 The effect of duality on motivic periods is more subtle, and requires some more definitions. In \cite{BrSVMP}, (2.20), we defined a canonical homomorphism
 \begin{equation} \label{pimottodR} 
 \pi= \pi^{\mathfrak{u}, \mm+} : P_{\MT(\Z)}^{\mm,+ } \To \Or(U^{dR})
 \end{equation}
 where $P_{\MT(\Z)}^{\mm,+}\subset P_{\MT(\Z)}^{\mm}$ denotes the subspace of effective motivic periods. The kernel of  $\pi$ is the ideal generated by 
 $\mathbb{L}^{\mm}$, the motivic version of $2\pi i$. There is an antipode
$$S :\Or(U^{dR})\To \Or(U^{dR}) , $$
which corresponds to   duality in the Tannakian category  $\MT(\Z)$.
Since, in a graded Hopf algebra there is a recursive formula for the antipode in terms of the coproduct, the map $S$  can be computed explicitly on the level of  unipotent de Rham versions  of motivic multiple zeta values (see \cite{BrSVMP}, \S2.4).
In particular, we have
\begin{equation}  \label{Sofoddzeta}
S(\zetau(2n+1))=-\zetau(2n+1)\ , \end{equation}
where $\zetau$ is the image of $\zetam$ under the map $(\ref{pimottodR}).$

The periods of $m(A,B)$ and $m(B,A)$ are therefore related by passing to unipotent de Rham periods via the map $\pi$, which kills
$\zetam(2)$, and applying the antipode $S$.
To state this cleanly we make some simplifying assumptions. Let   $A,B$ be boundary divisors  on  $\overline{\Mod}_{0,S}$ with no  common components, and let $\ell=|S|-3$. Suppose that 
$$\omega \in   m(A,B)_{dR} \qquad \hbox{ and } \qquad \omega' \in  m(B,A)_{dR} $$
$$X \in   m(A,B)_{B}^{\vee}  \qquad \hbox{ and } \qquad X' \in   m(B,A)_{B}^{\vee} $$
and to simplify matters, let us assume that 
$$\gr^W_0 m(A,B) \cong \Q(0) \qquad \hbox{ and } \qquad \gr^W_{2\ell} m(A,B) \cong \Q(-\ell)\ .$$
Then the same is true for $m(B,A)$, by duality $(\ref{mABdual})$.  For any object $M\in \MT(\Z)$ satisfying $W_{-1}M=0$,
we defined in \cite{BrSVMP}, (2.21) a map of $\Q$-vector spaces
$$c_0^t: M_B^{\vee} \To  M_{dR}^{\vee}$$
 as the dual of the map $c_0: M_{dR}=\oplus \gr^W_{2k} M_{dR} \rightarrow \gr^W_0 M_{dR} \cong \gr^W_0 M_{B} = W_0 M_B \subset M_B$.  With the above assumptions, the  classes of $\omega$ and $^tc_0(X')(-\ell)$ in $\gr^W_{2\ell} m(A,B)_{dR}$  differ by a rational number.
Therefore let $\alpha \in \Q$ such that 
 $$[^tc_0(X')(-\ell)]  =   \alpha\, [\omega] \in  \gr^W_{2\ell} m(A,B)_{dR}$$ 
  Likewise, let $ \alpha'\in \Q$ such that 
  $[^tc_0(X)(-\ell)] = \alpha'  \,  [\omega'] \in \gr^W_{2\ell} m(B,A)_{dR}$.

\begin{lem} With these assumptions, we have
$$  \alpha' \,  \pi  I^{\mm} (\omega, X) \  \equiv  \   \alpha \, S ( \pi (I^{\mm} (\omega', X'))) $$
where the equivalence means modulo the image under the map $\pi$  of motivic multiple zeta values of weight $\leq \ell-1$.
\end{lem} 
\begin{proof}
We have $I^{\mm}(\omega', X') = [m(B,A), \omega', X']^{\mm} $ and hence by \cite{BrSVMP}, (2.22), 
$$\pi [m(B,A), \omega', X']^{\mm} = [ m(B,A), \omega' , {}^tc_0(X')]^{\uu}$$
The antipode $S$ on matrix coefficients $[M, v_1,v_2]^{\uu}$ is $[M^{\vee}, v_2, v_1]^{\uu}$, so we have
$$ S \pi I^{\mm}(\omega', X')  = [m(B,A)^{\vee},   {}^tc_0(X'), \omega']^{\uu} = [m(A,B)(\ell),   {}^tc_0(X'), \omega']^{\uu} $$
Now $[V(r), v_1(r), v_2(r)]^{\uu} = (\mathbb{L}^{\uu})^r [V, v_1,v_2]^{\uu}$, and since $\mathbb{L}^{\uu} =1$, we have
 $[V, v_1,v_2]^{\uu} = [V(r), v_1(r), v_2(r)]^{\uu}$ for all $r\in \Z$.
Since $\pi  I^{\mm} (\omega, X)  = [m(A,B), \omega, {}^tc_0X)]^{\uu}$,   the statement follows. 
\end{proof} 

In  other words, the highest weight part of $I^{\mm}(\omega, X)$ is related, modulo $\zetam(2)$, to the highest weight part of $I^{\mm}(\omega',X')$
via the antipode on unipotent de Rham periods. 

\begin{rem} In the case when the motive $m(A,B)$ is self-dual, these observations give some non-trivial constraints on the periods which can occur.
 For example, via the equation 
 $S(\zetau(3,5)) = \zetau(3,5) +5 \zetau(3) \zetau(5),$
 we see that $\zeta(3,5)$ can never occur as a  period of a self-dual motive.
 Therefore  the self-dual cellular values for $\Mod_{0,11}$  (which we expect to be periods of self-dual motives) should evaluate to polynomials in single zeta values only.
 \end{rem}
 
 \section{Appendix 1:  A short compendium of integrals}
 The literature which has grown out of Ap\'ery's irrationality proofs for $\zeta(2)$ and $\zeta(3)$, and in particular, Beuker's interpretation  (found independently by Cordoba)
 using elementary integrals  \cite{Beu}, is vast. I have selected a very incomplete list of integrals with  various irrationality applications
 and reproduced them here in their original notations. The integrals  below are referred to in the main text, but there are many others that could also
  have been included.

 \subsection{Beukers' integrals for $\zeta(2)$ and $\zeta(3)$}
 The following family of integrals:
 \begin{equation} \label{AIBeuk2}
  \int_{0}^1 \!\int_0^1 {x^n(1-x)^ny^n(1-y)^n  \over (1-xy)^{n+1}} {dx dy }   \qquad {\text{\cite{Beu}, Eqn.} (5)} 
 \end{equation} 
 for $n\geq 0$,   are linear forms in $1$ and $\zeta(2)$, and give exactly Ap\'ery's proof of the irrationality of $\zeta(2)$. 
 In \cite{Beu}, Beukers introduces the following family of integrals
 \begin{equation}  \label{AIBeuk3}
  \int_{0}^1\! \int_0^1 \!\int_0^1 {x^n(1-x)^ny^n(1-y)^n  w^n(1-w)^n \over (1-(1-xy)w)^{n+1}} {dx dy dw} \qquad \text{\cite{Beu}, Eqn. } (7)  
 \end{equation} 
 and proves that they give linear forms in $1$ and $\zeta(3)$, identical to those considered by Ap\'ery, and hence leads to
 the irrationality of $\zeta(3)$ (\cite{Beu}, \cite{FiBourbaki} \S1.3).
 \subsection{Rhin and Viola's generalisations to several parameters}
 In \cite{RV1},  Rhin and Viola consider a generalisation of $(\ref{AIBeuk2})$ depending on parameters $h,i,j,k,l\geq 0$
 \begin{equation}  \label{AIRV2}
  \int_{0}^1 \!\int_0^1 {x^h(1-x)^i y^j(1-y)^k  \over (1-xy)^{i+j-l}} {dx dy \over 1-xy } 
 \end{equation} 
 which give linear forms in $1, \zeta(2)$. These integrals had previously been  considered by Dixon  \cite{Dixon} in 1905.

 In \cite{RV2}, Rhin and Viola consider a family of integrals generalising $(\ref{AIBeuk3})$  which depend on parameters $h,j,k,l,m,q,r,s\geq 0$:
  \begin{equation}  \label{AIRV3}
  \int_{0}^1 \!\int_0^1 \!\int_0^1  {x^h(1-x)^l y^k(1-y)^s z^j (1-z)^q   \over (1-(1-xy)z)^{q+h-r}} {dx dy dz \over 1-(1-xy)z } 
 \end{equation} 
subject to the conditions $j+q = l +s$ and $m= k+r-h$. This family yields linear forms in $1$, $\zeta(3)$.
The families $(\ref{AIRV2})$ and $(\ref{AIRV3})$, combined with the group method initiated in the same papers  yield
the best irrationality measures for $\zeta(2)$ and $\zeta(3)$ which are presently known  (see  \cite{RV2}, \cite{FiBourbaki} \S3.1).

\subsection{Sorokin's integrals in $\zeta(2n)$}
In \cite{Sorokin}  Sorokin  considers the integrals 
\begin{equation}\label{AISo1} 
 \int_0^1\! \ldots \!\int_0^1 \prod_{j=1}^n { u_j^n (1-u_j)^n v_j^n(1-v_j)^n \over  ({1 \over u_1v_1\ldots u_{j-1}v_{j-1}} -u_jv_j)^{n+1}} du_j dv_j
 \end{equation} 
and proves that they give linear forms in even zeta values to deduce a new proof of the transcendence of $\pi$. By clearing the terms in the denominator
and renaming variables in accordance with \cite{Fi} $(7)$, one obtains the family for $n=2p$ even:
\begin{multline} \label{AISo2}
\int_{[0,1]^N}  { (y_1y_2)^{p(N+1)-1} (y_3y_4)^{(p-1)(N+1)-1} \ldots (y_{N-1}y_N)^N 
\over  \prod_{k\in \{2,\ldots, N\} \hbox{ even}}  (1-y_1y_2\ldots y_k)^{N+1} } \\ \times 
\prod_{k=1}^N (1-y_k)^{N} dy_1 \ldots dy_N\ .
\end{multline}
This family of integrals are periods of  the moduli space $\Mod_{0,2n+3}$ of the form $(\ref{GeneralCubicalInt})$.
I do not know if this family of integrals can be written as special cases of generalised cellular integrals on $\Mod_{0,2n+3}$.

\subsection{Rivoal and Fischler's integrals for odd zeta values}
Rivoal's linear forms \cite{RiCRAS} are equivalent to the following  family of integrals:
\begin{equation} \label{AIRiv1}
\int_{[0,1]^{a+1}}  {\prod_{i=0}^a x_i^{rn} (1-x_i)^n \over (1-x_0 \ldots x_{a})^{(2r+1)n+2} } dx_0 \ldots dx_{a} 
\end{equation}
where $n\geq 0$ and $a, r\geq 1$ such that $(a+1)n > (2r+1)n+2$.  He proves in particular  that if $n$ is even and  $a$ is odd $\geq 3$ then it gives
 linear  forms in odd zeta values  $1,\zeta(3)$, \ldots, $\zeta(a)$ and goes on to deduce that infinitely many of them are irrational.
This integral has weight drop in the sense that it is an $a+1$-fold integral whose periods are of weight at most $a$. I did not consider weight-drop  integrals here, although  the apparent simplicity of 
$(\ref{AIRiv1})$ suggests that it would  be  interesting to do so.

Instead, at the end of section 2.4 in \cite{FiBourbaki}, Fischler gives a variant of the above integrals (which are very well-poised, as opposed to simply well-poised), by multiplying the integrand of $(\ref{AIRiv1})$ by $(1+x_0\ldots x_a)/(1-x_0\ldots x_a)$  (\cite{FiBourbaki}, \S2.3.1). 
He proves that the latter  integrals are equivalent  (with slightly different notation) to:
\begin{equation}  \label{AIRiv2} 
 \int_{[0,1]^{a-1}} {\prod_{j=1}^{a-1} y_j^{rn}(1-y_j)^n dy_j \over (1-y_1y_2\ldots y_{a-1})^{rn+1} 
\prod_{2\leq 2j  \leq a-2 }  (1-y_1y_2\ldots y_{2j})^{n+1} }
\end{equation}
where $n\geq 0$,   $a\geq 3$ and $1 \leq r <{a\over 2}$ are integers.  In the case when $a$ is even, it gives linear forms in the odd 
zeta values $1,\zeta(3), \ldots, \zeta(a-1)$ (\cite{FiBourbaki}, proposition 2.5).  The relationship between
$(\ref{AIRiv1})$ and $(\ref{AIRiv2})$ is discussed in the two paragraphs preceding $\S3$ of \cite{FiBourbaki}  and builds on theorem 5 in \cite{Zu2}.  
See the discussion below. Note that when $a$ is odd, $(\ref{AIRiv2})$ apparently gives linear forms in even zeta values $1, \zeta(2), \ldots, \zeta(a-1)$.

\subsection{Generalisations} 
Some generalisations of Rivoal's integrals $(\ref{AIRiv1})$ to a three-parameter family of integrals
yielding linear forms in $1,\zeta(2), \ldots, \zeta(n)$ are given in \cite{RiSelberg}, theorem 1.
In   \cite{Zu2} equation (70), Zudilin  considers the family of integrals 
\begin{equation} \label{AIZud}
J_k(a_0,\ldots, a_k, b_1,\ldots, b_k) = \int_{[0,1]^k} {  \prod_{j=1}^k x_j^{a_j-1} (1-x_j)^{b_j-a_j-1} \over Q_k(x_1,\ldots, x_k)^{a_0} } dx_1\ldots dx_k
\end{equation}
generalising work of Vasilyev and Vasilenko. Here,  $k \geq 4$ and 
$$Q_k(x_1,\ldots, x_k) =  1 - x_1(1-x_{2}( 1-\cdots (1-x_k))) \ .$$
In \cite{Zu2},  theorem 5,  he relates a certain sub-family of these integrals to hypergeometric series, and proves as a consequence  that  if
 $$b_1 +a_2= b_2+a_3=\ldots =b_{k-1} +a_k$$
 then the integrals $J_k(a,b)$ yield linear forms in odd  zeta values when $k$ is odd, and even zeta values when $k$ is even.
 For example, when $k=5$, this gives a $6$-parameter family of integrals which are linear forms in $1, \zeta(3), \zeta(5)$ (note that our generalised cellular
 integrals for $\pi^{8}_{\mathrm{odd}}$ apparently yields an $8$-parameter family which is strictly bigger, with the same property). 
A version of this family of integrals is considered by Fischler in    \cite{Fi} (5). His family of 
 integrals is denoted by 
\begin{equation} \label{AIFisch1}
 I(a_1,\ldots, a_n, b_1, \ldots, b_n, c) = \int_{[0,1]^n} { \prod_{k=1}^n x_k^{a_k} (1-x_k)^{b_k} \over  \delta_n(x)^c  } {dx_1\ldots dx_n \over  \delta_n(x)} 
 \end{equation}
where he writes $\delta_n(x) $ for $Q_n(x_n,\ldots, x_1)$ and is clearly equivalent to $(\ref{AIZud})$.
These families of  integrals
 are not obviously of moduli space type. 
 
 However,  in \cite{Fi}  equation $(9)$, Fischler defines the family of integrals
\begin{equation} \nonumber
 K( A_1, \ldots, A_n, B_1,\ldots, B_n, C_2, \ldots, C_n ) = \int_{[0,1]^n} { \prod_{k=1}^n y_k^{A_k} (1-y_k)^{B_k} \over 
  \prod_{k=2}^n (1-y_1\ldots y_k)^{C_k+1}} dy_1\ldots dy_n 
  \end{equation}
which are evidently  period  integrals on $\Mod_{0,n+3}$
written in cubical coordinates $y_1,\ldots, y_n$. By applying a carefully-constructed change of variables, he proves that the 
$K(A,B,C)$ can be re-expressed as integrals of the form 
\begin{equation}  \label{AIFisch2}
  \int_{[0,1]^n} { \prod_{k=1}^n x_k^{\widetilde{a}_k} (1-x_k)^{\widetilde{b}_k} \over  \prod_{k=2}^n Q_k(x_n,\ldots, x_{n+k-1})^{\widetilde{c}_k}  } {dx_1\ldots dx_n} 
  \end{equation} 
in new parameters $\widetilde{a}, \widetilde{b}, \widetilde{c}$ expressible in terms of the $A,B,C$. 
In addition he shows that  the families of integrals
 $(\ref{AIFisch1})$, and hence $(\ref{AIZud})$, form a sub-family of the integrals  $K(A,B,C)$.
 Thus all the integrals considered in this section are in fact equivalent to periods of moduli spaces $\Mod_{0,n}$.
Both Fischler and Zudilin construct symmetry groups for their respective families of integrals $(\ref{AIZud})$  and $(\ref{AIFisch1})$, similar to those
introduced by Rhin and Viola \cite{RV1,RV2}.

 \section{Appendix 2: Examples of basic cellular integrals}
 
 \subsection{Convergent configurations}
 Let $\mathcal{C}_N$ denote the number of convergent configurations of size $N$. Then we find that
 $$
\begin{array}{|r|cccccccc|}
\hline
 N &    4& 5 & 6 & 7 & 8 & 9 & 10 & 11    \\ \hline
\mathcal{C}_N &  0  & 1  & 1 & 5 & 17 & 105 & 771 & 7028 \\
\hline
\end{array}
$$
Here follows   a  list of convergent configurations of size $N$, where $4\leq N \leq 8$.

 \subsubsection{$N=5$}  \label{ExN5} There is a unique convergent configuration:
 $$ {}_5 \pi={}_5 \pi^{\vee}= [5,2,4,1,3]$$

 \subsubsection{$N=6$} \label{ExN6}  There is a unique convergent configuration:
 $$ {}_6 \pi={}_6 \pi^{\vee}= [6,2,4,1,5,3]$$

 \subsubsection{$N=7$} There are five convergent configurations. There are two pairs of configurations and their duals:
 $$
\begin{array}{ccc}
  {}_7 \pi_1 = [7,2,4,1,6,3,5]  & \quad  , \quad  &  {}_7 \pi_1^{\vee} = [7,2,5,1,4,6,3]    \\
   {}_7 \pi_2 = [7,2,4,6,1,3,5] & \quad  , \quad  &  {}_7 \pi_2^{\vee} = [7,3,6,2,5,1,4]    
  \end{array}
  $$
  and a single self-dual configuration:
$$     {}_7 \pi_3 ={}_7 \pi_3^{\vee}= [7,2,5,1,3,6,4]     $$

 \subsubsection{$N=8$} There are 17 convergent configurations, comprising 7 pairs of configurations and their duals:
 $$
\begin{array}{ccc}
    {}_8\pi_1=  [8, 2, 4, 1, 5, 7, 3, 6]   & \quad  , \quad  &  {}_8   \pi^{\vee}_{1}=[8, 2, 5, 1, 7, 4, 6, 3]    \\
    {}_8\pi_4=     [8, 2, 4, 7, 1, 6, 3, 5]   & \quad  , \quad  &   {}_8\pi_4^{\vee}=      [8, 2, 4, 7, 3, 6, 1, 5]    \\
     {}_8  \pi_{5}=    [8, 2, 5, 3, 7, 1, 6, 4]  & \quad  , \quad  &      {}_8\pi_{5}^{\vee}=      [8, 2, 6, 1, 5, 3, 7, 4]    \\
     {}_8   \pi_7=     [8, 2, 4, 6, 1, 3, 7, 5]   & \quad  , \quad  &      {}_8  \pi_{7}^{\vee}=    [ 8, 2, 5, 1, 6, 3, 7, 4]     \\
    {}_8     \pi_{8}= [ 8, 2, 5, 1, 6, 4, 7, 3]    & \quad  , \quad  &      {}_8    \pi_8^{\vee}=  [8, 2, 4, 1, 7, 5, 3, 6]   \\
  {}_8      \pi_{9}=  [8, 2, 5, 7, 3, 1, 6, 4]   & \quad  , \quad  &      {}_8     \pi_{9}^{\vee}=  [8, 3, 6, 1, 5, 2, 7, 4]    \\
 {}_8      \pi_{10}=  [8, 2, 5, 7, 3, 6, 1, 4] & \quad  , \quad  &      {}_8       \pi_{10}^{\vee}=  [8, 2, 5, 7, 4, 1, 6, 3]  \\
  \end{array}
  $$
  and three self-dual configurations:
\begin{eqnarray}
{}_8\pi_2= {}_8\pi_2^{\vee} & =  &    [8, 2, 4, 1, 6, 3, 7, 5 ]  \nonumber\\
   {}_8\pi_{3}={}_8\pi_3^{\vee} & =  & [8, 2, 5, 1, 7, 3, 6, 4] \nonumber \\
   {}_8\pi_{6} ={}_8\pi_{6}^{\vee} & =  &  [8, 3, 6, 1, 4, 7, 2, 5] \nonumber 
\end{eqnarray} 
  
  \subsection{Basic cellular integrals}
  \subsubsection{n=5}
 In simplicial  coordinates $(t_1,t_2)$, and $\sigma= (5,2,4,1,3)$ we have
 \begin{equation} 
 f_{\sigma}  =   { t_1 (t_1-t_2)(t_2-1) \over (t_1-1)t_2} \quad \hbox{ and } \quad 
 \omega_{\sigma} =    { dt_1 dt_2 \over (t_1-1)t_2} \nonumber
 \end{equation} 
 From theorem $\ref{thmperiodsofmon}$, for example, we know that   $I_{\sigma}(N)$ is a linear form in $1$ and $\zeta(2)$. 
  Furthermore, we verify that
  $$I_{\sigma}(N)= \int_{\Dom_5} f_{\sigma}^N \omega_{\sigma} =  a_N \zeta(2) -  b_N$$
  where $a_N,b_N$ are solutions to the recurrence A005258 in \cite{Sloane}
  \begin{equation}  \label{Ap2} (N+1)^2 u_{N+2}   -(11 N^2+11 N+3) u_{N+1}- N^2 u_N=0\ . \end{equation}
 with initial conditions $a_0 =1, a_1 =3$, $b_0 =0, b_1=5$. This is   precisely Ap\'ery's sequence for $\zeta(2)$.  It is self-dual; i.e., the polynomial 
 $p(N) = 11 N^2 +11N+3$ satisfies 
 $p(-1-N) =p(-N)$ 
(equivalently, the coefficient of $N^2$ equals the coefficient of $N$).

\begin{rem} Changing to cubical coordinates  via $t_1=xy$, $t_2=y$, we get
  $$I_{\sigma}(N) = \int_{0}^1 \int_0^1 \Big( {x y (1-x) (1-y) \over (1-xy)}\Big)^N {dx dy \over 1-xy}$$
  which, is exactly Beuker's integral for  $(\ref{AIBeuk2})$. \end{rem}

  \subsubsection{n=6}There is again a  unique convergent configuration up to symmetry, namely $\sigma = (6,2,4,1,5,3)$. In 
simplicial coordinates $(t_1,t_2,t_3)$, we have
\begin{equation}
f_{\sigma} = {t_1 (t_2-t_1)(t_3-t_2)(t_3-t_1)  \over t_2(t_1-1)(t_2-1)t_3} \quad    \hbox{ and } \quad  \omega_{\sigma} = {dt_1 dt_2 dt_3 \over t_2(t_1-1)(t_2-1)t_3} 
\end{equation} 
It follows, for example, from theorem \ref{thmperiodsofmon} that $I_{\sigma}(N)$ is a linear form in $1$ and $\zeta(3)$ only, i.e., the coefficient of $\zeta(2)$
that could have occurred by theorem $\ref{thmperiodsofmon}$  vanishes. Furthermore
$$I_{\sigma}(N) = \int_{0\leq t_1 \leq t_2 \leq t_3 \leq  1}  f^N \omega =2\, a_N \zeta(3) -  b_N$$
where $a_N, b_N$ 
satisfy the recurrence (A005259 in \cite{Sloane})
\begin{equation} \label{Ap3} (N+1)^3 u_{N+2}  -(2N+1)(17 N^2+17 N+5) u_{N+1} + N^3 u_{N} = 0 \ . \end{equation}
which is precisely Ap\'ery's sequence for $\zeta(3)$, with initial conditions
$a_0=1, a_1=5$ and $b_0=0,b_1=6$. This equation  is again self-dual, which implies that the coefficient of $u_{N+1}$ is  of the form $(2N+1)(aN^2+aN+b)$ for some $ a, b\in \Q$.
By passing to cubical coordinates $t_1=xyz, t_2 = yz, t_3=z$ and applying the change of variables $(\ref{ChangeofVarRV3})$, we see that this family of integrals
exactly coincides with Beuker's integrals $(\ref{AIBeuk3}).$ This follows on setting all parameters in $(\ref{gencellz3})$ equal to $N$.

  \subsubsection{n=7}
  For $\sigma$ one of the  five  convergent configurations ${}_7\pi_i$, for  $i=1,\ldots, 5$  listed above, the integrals $I_{\sigma}(N)$ are  distinct linear forms  
$$a^{}_N \zeta^{\sigma} + b^{}_N \zeta(2) + c^{}_N$$
for some $I_{\sigma}(0) = \zeta^{\sigma} \in \Q^{\times} \zeta(4)$, and $a^{}_N$, $b^{}_N$, $c^{}_N$   are solutions to an equation 
$$   p^{(i)}_3 u_{N+3} + p^{(i)}_2 u_{N+2}+ p^{(i)}_1 u_{N+1}+  p^{(i)}_0 u_{N}=0$$
 where $p^{(i)}_j$ are polynomials of degree $6$ in $N$. 
 The   sequences  $a_N$ appear to satisfy interesting congruence properties  along the lines of \cite{Osb}.

The numbers which are observed to  occur as basic cellular integrals (or indeed as generalised cellular integrals)
 are indicated in the table below.  A dark $\bullet$ indicates that the corresponding period  can occur with non-zero coefficient, a $0$ indicates
  that the corresponding period is not observed to occur in the generalised cellular integrals.
  The sign of $I_\pi(0)$ has been chosen to make the integral positive.

$$
\begin{array}{|r|c|c|c|c|c|}\hline
 \hbox{Configurations}                    &  1   &  \zeta(2) & \zeta(3) & \zeta(4) &    I_{\pi}(0) \\ \hline
                   {}_7\pi_1 & \bullet & \bullet & 0  & \bullet &   {17\over 10} \zeta(2)^2\\
                                 \hline
         
                   {}_7\pi_2 & \bullet & \bullet & 0  & \bullet &   {27\over 10} \zeta(2)^2\\
                                 \hline
                   {}_7\pi_3 & \bullet & \bullet & 0  & \bullet &   \zeta(2)^2 \\
                                 \hline
                   {}_7\pi^{\vee}_1 & \bullet & \bullet & 0  & \bullet &  {7\over 10} \zeta(2)^2\\
                                 \hline
                   {}_7\pi^{\vee}_2 & \bullet & \bullet & 0  & \bullet &    {3\over 10} \zeta(2)^2\\
                                 \hline
                                           \end{array}
$$
  \vspace{0.1in}

 For ${}_7\pi_3 = (7, 2, 5, 1, 3, 6, 4)$, the basic cellular integrals satisfy
$I_n^{{}_7 \pi_3} = \big(I^{{}_5\pi_1}_n\big)^2$, the square of the Ap\'ery sequences for $\zeta(2)$.
The configuration $\pi^7_{\mathrm{even}}$ of \S\ref{sectOddZ} is   ${}^7\pi_1^{\vee}$.

  \subsubsection{n=8}
  
  Experimentally, we find that the   17 convergent configurations  only give rise to  13 distinct families of linear forms, given in the table below.

  $$
\begin{array}{|r|c|c|c|c|c|c|c|}\hline
 \hbox{Configurations}                    &  1   &  \zeta(2) & \zeta(3) & \zeta(4) & \zeta(5) & \zeta(3)\zeta(2)  & I_{\pi}(0) \\ \hline
                   {}_8\pi_1 \ , \     {}_8\pi^{\vee}_{1}  & \bullet &\bullet &\bullet& 0& 0 & \bullet & 2 \zeta(2) \zeta(3)\\
                                 \hline
                      {}_8\pi_2 \ , \ {}_8\pi_3^{\vee} & \bullet &\bullet &\bullet&  0 & \bullet & \bullet & \zeta(5)+ \zeta(3)\zeta(2)\\
                                                                                   \hline
                                        {}_8\pi_4  \ , \   {}_8\pi_{5}  & \bullet &\bullet &\bullet&  0 & \bullet & \bullet & 9\zeta(5)-2\zeta(2)\zeta(3)\\
                                                   \hline
                                        {}_8\pi_4^{\vee} \ , \ {}_8\pi_{5}^{\vee}  & \bullet &\bullet &\bullet&  0 & \bullet & \bullet& 9\zeta(5)-4\zeta(3)\zeta(2)\\
                                          \hline
                     {}_8 \pi_{6}  & \bullet &\bullet &\bullet &  0 & \bullet & \bullet & 16 \zeta(5)-8\zeta(3)\zeta(2)\\
                      \hline
                                        {}_8\pi_7   & \bullet &\bullet &\bullet&  0 & \bullet & \bullet & \zeta(5)+3\zeta(3)\zeta(2)\\
                     \hline
                                        {}_8\pi_{7}^{\vee}   & \bullet &\bullet &\bullet&  0 & \bullet & \bullet & \zeta(3)\zeta(2)-\zeta(5) 
                              \\
                \hline
                                       {}_8 \pi_{8}    & \bullet &0 &\bullet&  0 & \bullet & 0 & 2\zeta(5)\\  
                                                  \hline
                     {}_8 \pi_8^{\vee}  & \bullet &\bullet & 0  &  0 & \bullet & \bullet & 2\zeta(5)+ 4\zeta(3)\zeta(2)\\
                    \hline
                     {}_8 \pi_{9}  & \bullet &\bullet &\bullet &  0 & \bullet & \bullet & 6 \zeta(3)\zeta(2) -7 \zeta(5)  \\
\hline
                  {}_8 \pi_{9}^{\vee}  & \bullet &\bullet &\bullet &  0 & \bullet & \bullet&  4\zeta(3)\zeta(2) - 7 \zeta(5) \\
                    \hline
                     {}_8 \pi_{10} & \bullet &\bullet &\bullet &  0 & \bullet & \bullet & 5\zeta(3)\zeta(2) -8 \zeta(5)\\
                    \hline
                    {}_8  \pi_{10}^{\vee}& \bullet &\bullet &\bullet &  0 & \bullet & \bullet & 8\zeta(5)-3\zeta(3)\zeta(2)\\
                                     \hline
                  \end{array}
$$
  \vspace{0.1in}

The  values of the  integrals $I_{\pi}(0)$ are consistent with Poincar\'e-Verdier duality:
$$|I_{\pi}(0)| \equiv |I_{\pi^{\vee}}(0)| \pmod{\zeta(2)}$$ 
 as expected, since the antipode acts by $-1$ on single odd zeta values    $(\ref{Sofoddzeta})$.
The first set of entries of the table gives the series which is the product of the Ap\'ery sequences           
 for $\zeta(2), \zeta(3)$ via the multiplicative structure \S\ref{sectMultStruct} (example \ref{exampleprod}) 
 $$ I_{{}_8\pi_1}(N) = I_{{}_5\pi_1}(N)  I_{{}_6\pi_1}(N)$$
   There is a unique entry, ${}_8 \pi_8 = {}^8 \pi_{\mathrm{odd}}$ which  gives a linear form in $1, \zeta(3), \zeta(5)$  only.
   It is amply sufficient to prove that $\dim_{\Q}  \langle1, \zeta(3), \zeta(5) \rangle_{\Q} \geq 2$
but insufficient to prove their linear independence. The dual ${}_8 \pi_8^{\vee}$ of this sequence  gives linear forms in $1, \zeta(2)$ and $2 \zeta(5)+ 4\zeta(3)\zeta(2)$ from which we can extract linear forms in $1, \zeta(5)$ (see \S\ref{sectdualforms}).

\subsubsection{n=9} There are 105 convergent configuration classes. By computing all of them in low degrees, one observes that 
all  the linear forms vanish in weight $5$, and so in particular,  the  coefficient of $\zeta(5)$  and $\zeta(2)\zeta(3)$  always vanishes.
 All possible products of previously occurring sequences arise, namely,  the product of the canonical
sequence for $\Mod_{0,5}$ with one of the five sequences for $\Mod_{0,7}$, and the square of the canonical sequence
for $\Mod_{0,6}$.
Other than that there 
are five  configurations 
$$
   [9, 2, 4, 1, 8, 6, 3, 5, 7]  \ , \   [9, 2, 4, 6, 8, 1, 3, 5, 7] \ , \  [9, 2, 5, 8, 1, 4, 7, 3, 6] $$
   $$
  [9, 2, 6, 1, 5, 7, 4, 8, 3]  \ ,    [9, 4, 8, 3, 7, 2, 6, 1, 5]  $$
which  give distinct irreducible sequences (i.e., not reducing to products of previously-occuring sequences) 
which are new linear forms in $1, \zeta(2), \zeta(2)^2, \zeta(2)^3$.
Generalising such families may lead to new estimates for the transcendence measure of $\pi^2$.

Finally, there are exactly four self-dual configurations 
$$[9, 2, 4, 1, 5, 7, 3, 8, 6]\ , \ [9, 2, 4, 1, 5, 8, 6, 3, 7] \ , \ [9, 2, 4, 6, 1, 7, 5, 8, 3]\ , \ [9, 2, 4, 7, 5, 1, 6, 8, 3] $$
which give (the same) linear forms in $1, \zeta(3), \zeta(3)^2$.

\subsubsection{n=10}
 There are 771 convergent configurations. They all (experimentally) vanish  in sub-leading weight.
Many cases are products of previously occurring sequences.
We find some new phenomena:
\begin{enumerate}[itemsep=2pt,parsep=0pt]
\item {\it Odd zeta values only}. A unique   configuration  $\pi^{10}_{\mathrm{odd}}$ 
which gives rise to a family of linear forms in $1 ,\zeta(3), \zeta(5), \zeta(7)$. It vanishes in weights $1,2,4,6$.

\item  {\it Dual of the previous case}. The  configuration $(\pi^{10}_{\mathrm{odd}})^{\vee}$  gives a family of linear forms in $1, \zeta(2), \zeta(4), \zeta_7$
where $\zeta_7$ denotes a multiple zeta value of weight 7. It  vanishes in weights $1,3,5,6$.

\item  {\it Double vanishing at next to leading order}. Two families of linear forms (represented  by $(10, 2, 4, 1, 6, 8, 5, 3, 9, 7)$  
for the first, and  its dual for the second)
which  give  linear forms in  quantities
$$1, \zeta(3), \zeta(4) \hbox{ and } \zeta_7$$
where $\zeta_7$ is an MZV  of weight $7$. These  families 
 vanish in weights $1,2,5, 6$.
\item  {\it Vanishing in the middle}. A unique family of self-dual linear forms  (represented by  $(10,2,4,1,6,3,8,5,9,7)$), 
 which give rise to linear forms  in the quantities
$$ 1, \zeta(2),  \zeta_5 \hbox{ and }\zeta_7 $$
where $\zeta_5$ and $\zeta_7$ are certain multiple zeta values of weights $5$ and $7$ respectively. Therefore this family 
vanishes in weights $1,3,4,6$.
\end{enumerate}

These (experimental) examples show that the vanishing phenomena  can be quite diverse.
It would be interesting to know if it is possible to find sequences of higher order with stronger vanishing properties.

\section{Appendix 3}

  The purpose of this section is to prove some vanishing of cohomology of the moduli space motives $m(A,B)$ in certain cases  which covers  Ap\'ery's theorems.

\subsection{Statement}

\begin{defn} Let $A\subset \overline{\Mod}_{0,S}$ denote a  boundary divisor. Say that $A$ is  \emph{cellular}
if there exists a dihedral structure $\delta$ on $S$ such that the irreducible components of $A$ are exactly the divisors at finite distance with respect to $\delta$.
Equivalently, 
$$A  = \bigcup_{S=S_1 \cup S_2} D_{S_1|S_2}$$
where the union is over all stable partitions of $S$ with $S_1, S_2$ consecutive for $\delta$.
\end{defn} 
A divisor $A$ is cellular if it is the Zariski closure of the boundary of a cell $\Dom_{\delta}$ for some $\delta$. 
The following theorem states that cellular divisors always fulfil cohomological vanishing in sub-maximal weights.

\begin{thm}\label{thmCohomVanish} Suppose that $A,B \subset \overline{\Mod}_{0,n}$ are   cellular boundary divisors  with no common irreducible  components.  Let $\ell = n-3$. Then
\begin{equation}
  \gr^W_2 m(A,B)  =  \gr^W_{2\ell-2}  m(A,B)  =0 
  \end{equation}
  and $\gr^W_0 m(A,B)$ and $\gr^W_{2\ell} m(A,B)$ are both $1$-dimensional.
\end{thm}
In the case $n=5$ and $n=6$, there is a unique choice of divisors $A,B$ satisfying the conditions of the previous theorem, up to automorphisms of $\Mod_{0,n}$.
Denote the corresponding motives $m(A,B)$ by $m_5$ and $m_6$ respectively.  They  could be called  \emph{Ap\'ery motives} by the following corollary. 

\begin{cor} The basic cellular integrals for $n=5$ and $n=6$ ($\S\ref{ExN5},\S\ref{ExN6}$) are periods of $m_5, m_6$ respectively.  These  motives satisfy
\begin{eqnarray}
\gr^W_{\bullet} m_5  & \cong &  \Q(0) \oplus \Q(-2) \nonumber \\
\gr^W_{\bullet} m_6 &  \cong &  \Q(0) \oplus \Q(-3) \ .\nonumber 
\end{eqnarray}
In particular, this implies that the basic cellular integals for $n=5$ give linear forms in  $1, \zeta(2)$, and for $n=6$ in $1,\zeta(3)$ only.
\end{cor} 
\begin{proof} Use equations $(\ref{orderoff})$ and  $(\ref{ordDomega})$ to verify that for all $N \geq 0$,
$$\mathrm{Sing} (f_{\delta/\delta'}^N \omega_{\delta'})  = \mathrm{Sing} (\omega_{\delta'}) $$
is cellular for $n=5, 6$.   For $n=5$ this is trivial, for $n=6$ one must check the divisor of singularities using $(\ref{ordDineq})$ and $(\ref{ordDomega})$. Therefore the basic cellular integral $I_{\delta/\delta'}(N)$  is a period of $m(A,B)$ where
$A=  \mathrm{Sing} (\omega_{\delta'}) $, which is cellular by lemma \ref{lempolesform} and where $B$ is the Zariski closure of the boundary of the domain 
of integration $\Dom_{\delta}$, which is cellular by definition.
\end{proof} 

Unfortunately, the basic cellular integrals for $n\geq 7$ marked points have divisors which are not cellular, but are weakly 
cellular in the sense, for example, of remark \ref{remWeakCell}.
It would be interesting to generalise the previous theorem to cover this case (and even the case of generalised cellular integrals). This would explain the observed vanishing of subleading coefficients in all  cases.

  \subsection{Cohomology computations}
  
   We require the following general observations.  Let $X$ be a smooth projective scheme over a field $k$ of characteristic $0$, and $A\cup B$ a simple normal crossing divisor in $X$.
    Denote the irreducible components of $A$ and $B$ by $A_i$, $B_i$, respectively, and write 
  $C_{\emptyset} = X$, and $C_I = \cap_{i\in I} C_i$, whenever $C=A$ or $C=B$,  and  $I$ is an indexing set.  All cohomology has $\Q$ coefficients.
  
  The relative cohomology spectral sequence  is:  
  \begin{equation} \label{RelCohomSS}
   E_1^{p,q} = \bigoplus_{|I|=p} H^q (B_I \backslash (B_I \cap A) ) \ \Rightarrow \   H^{p+q}(X\backslash A, B \backslash (B\cap A))
   \end{equation}
  The other spectral sequence we need is the weight (or Gysin) spectral sequence
  \begin{equation} \label{GysinSS}
   E_1^{p,q} = \bigoplus_{|I|=-p} H^{2p+q} (A_I )(p) \ \Rightarrow \   H^{p+q}(X\backslash A)
   \end{equation}
  which degenerates at $E_2$ for reasons of purity.   The following lemmas give some control over the lowest graded weight pieces of moduli space motives.

We shall say that a boundary divisor $D \subset \overline{\Mod}_{0,S_1} \times \ldots \times \overline{\Mod}_{0,S_r}$ is at finite (resp. infinite) distance with respect to 
dihedral structures $\delta_1$ on $S_1$,  \ldots, $ \delta_r $ on $S_r$ if its irreducible components are of the form $ \overline{\Mod}_{0,S_1} \times \ldots  \times D_i \times \ldots  \times \overline{\Mod}_{0,S_r}$ where $D \subset \overline{\Mod}_{0,S_i}$ is at finite (resp. infinite) distance   with respect to $\delta_i$.

When $D$ is at finite distance, we say that $D$  is \emph{complete}
if its set of  irreducible components  consists of all divisors at finite distance, and \emph{broken} otherwise.

\begin{lem} \label{lemcompletebroken} Let $\delta_i$ be  a dihedral structure  on $S_i$,  for $i=1,2$, where $|S_i|\geq 3$. 
Let $C \subset \overline{\Mod}_{0,S_1} \times \overline{\Mod}_{0,S_2} $ be a boundary divisor  at finite distance with respect to $\delta_1,\delta_2$.  If $n=|S_1|+|S_2|-6>0$ then
$$\gr^W_0 H^n( \overline{\Mod}_{0,S_1} \times \overline{\Mod}_{0,S_2} , C) \cong \begin{cases} \Q(0) \quad \hbox{ if } C  \hbox{ is complete }  \\  0 \  \qquad \hbox{  if } C \hbox{ is broken }\end{cases}\ .$$
\end{lem}
\begin{proof} The relative cohomology spectral sequence, applied to $X=  \overline{\Mod}_{0,r} \times \overline{\Mod}_{0,s}$,  
$A= \emptyset$, $B=C$, degenerates at $E_2$ because $E_1^{p,q}$ is pure of weight $q$ in this case.  This 
 implies that  $\gr^W_0 H^n( \overline{\Mod}_{0,r} \times \overline{\Mod}_{0,s} , C)$ is the $n^{\mathrm{th}}$ cohomology group of the complex
\begin{equation}\label{grw0complex}
E_1^{\bullet, 0}  \ =  \ H^0(C_{\emptyset}) \rightarrow   \bigoplus_{|I|=1} H^0(C_I) \rightarrow \bigoplus_{|I|=2} H^0(C_{I})\rightarrow \cdots  \rightarrow \bigoplus_{|I|=n} H^0(C_{I})
\end{equation}
By assumption, the irreducible components of $C$ are in one-to-one correspondence with  a subset of the set of  facets of the product  $X_{\delta_1} \times X_{\delta_2}$
where  $X_{\delta}= \overline{\Dom}_{\delta}$ denotes the closure of a cell in the analytic topology. They have the combinatorial structure of Stasheff polytopes. Consider the simplical complex $S$ which is dual to the one generated by the facets of $X_{\delta_1} \times X_{\delta_2}$. Then the  $0$-simplices of $S$
are indexed by facets of $X_{\delta_1} \times X_{\delta_2}$, the 1-simplices by codimension 2 faces of $X_{\delta_1} \times X_{\delta_2}$, and so on.
Since each $X_{\delta_i}$ is homotopic to a ball, the same is true of $X_{\delta_1} \times X_{\delta_2}$, and therefore $S$ is homotopic to its boundary, which is a  sphere $\mathbb{S}^{n-1}$ of dimension $n-1$.
The complex $(\ref{grw0complex})$, shifted by one to the left,  computes the reduced cohomology of the simplical subcomplex $T\subset S$ whose $0$-simplices correspond to irreducible
divisors of $C$.  Therefore, if  $C$ is complete, then $T=S$ and  the $(n-1)^{\mathrm{th}}$ reduced cohomology group of $S$  is that of $\mathbb{S}^{n-1}$ and one-dimensional.
If $C$ is broken,  then $T\subsetneq S$ is a  simplicial subcomplex of a punctured $n-1$-sphere, and homotopic to a simplical complex in $\R^{n-2}$.  Therefore its $(n-1)^{\mathrm{th}}$ cohomology group vanishes.
\end{proof}

The following lemma provides a canonical basis for $H^2(\overline{\Mod}_{0,S})$ for every dihedral ordering $\delta$ on $S$.
\begin{lem} \label{lembasis} Let $|S_i| \geq 3$,  with   dihedral orderings $\delta_i$ on $S_i$, for $i=1,\ldots, r$. Then
$$H^2(\overline{\Mod}_{0,S_1} \times \ldots \times \overline{\Mod}_{0,S_r}) \cong \bigoplus_{D\in \delta_{\infty}}  [D]\Q(-1) $$
where $\delta_{\infty}$ is the set of irreducible divisors at infinite distance, and for each $D\in \delta_{\infty}$,  $[D]$ is the image of the canonical class of $H^0(D)$ under the Gysin 
map 
$$H^0(D)(-1) \rightarrow H^2(\overline{\Mod}_{0,S_1}\times \ldots \times \overline{\Mod}_{0,S_r})\ .$$
\end{lem} 

\begin{proof} Let $r=1$. By the results of Keel \cite{Keel}, the Gysin map
$$\bigoplus_{D \in \delta_f \cup \delta_{\infty}}  H^0(D)(-1)  \To H^2(\overline{\Mod}_{0,S}) $$
where the left-hand sum is over all irreducible boundary divisors $D$, is surjective, and its kernel is generated by
relations
\begin{equation}  \label{Keelrel}
\sum_{\{j,k\}\in S_1, \{i,l\}\in S_2}  [D_{S_1|S_2}]  =  \sum_{\{i,k\}\in S_1, \{j,l\}\in S_2}  [D_{S_1|S_2}]
\end{equation} 
for all sets  of four distinct elements $i,j,k,l \in S$. We first show that
\begin{equation} \label{DinftoH2mos} \bigoplus_{D \in  \delta_{\infty}}  H^0(D)(-1)  {\To} H^2(\overline{\Mod}_{0,S}) 
\end{equation} 
is surjective. For this,   choose $i,j,k,l\in S$ such that $i<j<k<l$ are in order with  respect to $\delta$ and the pairs $\{i,j\}$ and $\{k,l\}$ are consecutive,
and $\{j,k\}$ ,$\{l,i\}$ are not consecutive.  Then there is a single term in $(\ref{Keelrel})$ which corresponds to a divisor in $\delta_f$, namely
$D_{S_1|S_2}$ with $S_1 = \{j,j+1, \ldots, k\}, S_2 = \{l, l+1, \ldots, i\}$. This proves that the $[D]_{S_1|S_2}$ for $S_1, S_2$ consecutive are linear combinations of $[D]$ with $D \in \delta_{\infty}$. 
On the other hand, we know from \cite{Keel}, page 550,  that the dimension of   $H^2(\overline{\Mod}_{0,S}) $ is $2^{n-1} - \binom{n}{2} -1$, which 
by a straightforward calculation, is the cardinality of $\delta_{\infty}$, so $(\ref{DinftoH2mos})$ is injective.
The case when $r\geq 2$ follows from the K\"unneth formula. 
\end{proof}

The basis is functorial in the following sense.  Let  $S$ have a dihedral structure $\delta$, and let 
$B$ be an irreducible divisor  at finite distance on  $\overline{\Mod}_{0,S}$ corresponding to a stable partition $S_1\cup S_2 $ of $S$.   Then $B = \overline{\Mod}_{0,S_1 \cup x} \times \overline{\Mod}_{0,S_2 \cup x}$ and  $S_1 \cup x$ and $S_2 \cup x$ inherit canonical dihedral structures from $\delta$.
The restriction morphism in this basis is 
$$[D] \mapsto [D\cap B]: H^2(\overline{\Mod}_{0,S}) \To H^2(B)\ .$$

The key geometric property of cellular divisors is the following.
\begin{lem} \label{newlem} Let $C \subset \overline{\Mod}_{0,S}$ be a cellular boundary divisor and let
$D$ be  an irreducible boundary divisor such that $D\subsetneq C$ and  $C\cap D\neq \emptyset$.  Then  for all $k\geq 0$,
$$\gr^W_0 H^k (D, D\cap C) = 0\ . $$
\end{lem}
\begin{proof} Apply the relative cohomology spectral sequence with $X=D$, $A =\emptyset$ and $B=C\cap D$.
  It suffices to show that the complex
\begin{equation}\label{H0D} 
E_1^{\bullet,0 } = H^0(D) \rightarrow \bigoplus_i H^0(D \cap C_i) \rightarrow \bigoplus_{i,j} H^0(D\cap C_{i,j}) \rightarrow \ldots
\end{equation}
is acyclic.   Consider the simplicial complex $T$ whose $k$-dimensional simplices are given by codimension $k$ intersections of  irreducible components $(D\cap C)_I$, where $|I|=k$. Then $(\ref{H0D})$ computes the reduced homology of $T$.
Now $C$ is cellular with respect to a dihedral ordering $\delta$, and its irreducible components are indexed by  stable partitions  of $S$ which are consecutive with respect to $\delta$. Since $D\subsetneq C$, the divisor $D$ corresponds to a partition $V_1 \cup V_2$ of $S$ where $V_i$  are not consecutive with respect to $\delta$. The set  $S$ can be written as a disjoint union of subsets 
$S = I_1 \cup I_2 \cup \ldots \cup I_r$, where each $I_k$ is maximal such that its elements  are consecutive for $\delta$, and each $I_k$ is alternately contained in either $V_1$ or $V_2$.
The irreducible  components of $C\cap D$ are indexed by stable partitions $S= P \cup (S\backslash P)$ where $|P|\geq 2$ and $P$ ranges
over all consecutive subsets of  each $I_k$ (see example below). Choose a $k$ such that $|I_k|\geq 2$. Then the divisor  corresponding to $P=I_k$ intersects the divisors corresponding to all other $P_i$. It follows that the simplical complex $T$ is a cone with apex $P$, and is therefore contractible.
\end{proof}
\begin{example} Let $S= \{1,\ldots, 7\}$ and $\delta = \delta^0$. Then irreducible components of  $C$ are divisors corresponding to stable partitions $S_1, S_2$ of $S$
where $S_1, S_2$ consist of consecutive elements. Let $D$ be the divisor corresponding to the stable partition $\{1,4\} \cup \{2,3,5,6,7\}$. Then
write $S= \{1\} \cup \{2,3\} \cup \{4\} \cup \{5,6,7\}$, and  the  irreducible components of $D\cap C$
correspond to stable partitions  $P_i \cup (S\backslash P_i)$ where
$$P_1 = \{2,3\}   \quad , \quad P_2 = \{5,6\} \quad  ,  \quad P_3 = \{6,7\} \quad , \quad P_4 = \{5,6,7\}$$
 Denote by  $P_{ij} = P_i \cap P_j$, and $P_{ijk} = P_{ij} \cap P_k$. The simplical
complex  $T$   has vertices
 $P_1,\ldots, P_4$,  lines 
  $P_{12}, P_{13}, P_{14}, P_{24}, P_{34}$ and  2-faces  $P_{134}, P_{124}$.  Since $P_4$ meets all other $P_i$, the complex $T$ is a cone, with apex $P_4$, over the subcomplex generated by $P_1,P_2,P_3$.
\end{example}

\begin{lem} \label{lemgrw2vanishes} Let $C\subset \overline{\Mod}_{0,S}$ be a cellular boundary divisor with respect to $\delta$. Then 
$$\gr^W_2 H^{k+2}(\overline{\Mod}_{0,S}, C) = 0 \qquad \hbox{ for all } k>0$$
and $\gr^W_2 H^2(\overline{\Mod}_{0,S}, C)$ has dimension equal to the number of  irreducible divisors  $D \in \delta_{\infty}$  on $\overline{\Mod}_{0,S}$
which do not meet $C$.
\end{lem} 

\begin{proof} Apply the relative cohomology spectral sequence,  with $X=\overline{\Mod}_{0,S}$, $A=\emptyset$ and $B=C$.
It degenerates at $E_2$  by purity. Thus   $\gr^W_2 H^{k+2}(\overline{\Mod}_{0,S}, C)$ is the
 $k^{\mathrm{th}}$ cohomology group of the complex
$$ E_1^{\bullet,2} = H^2(\overline{\Mod}_{0,n}) \rightarrow \bigoplus_i  H^2(C_i) \rightarrow \bigoplus_{i,j}  H^2(C_{i,j}) \rightarrow \ldots $$
By  lemma $\ref{lembasis}$, this complex splits as a direct sum of complexes, one for each irreducible divisor $D \in \delta_{\infty}$. The previous complex is isomorphic to 
$$\bigoplus_{D \in \delta_{\infty}}\Big( H^0(D) \rightarrow \bigoplus_i H^0(D \cap C_i) \rightarrow \bigoplus_{i,j} H^0(D\cap C_{i,j}) \rightarrow \ldots \Big)\otimes \Q(-1)$$
This is a direct sum of complexes $(\ref{H0D})$, which are acyclic in the case that $C\cap D \neq \emptyset$ by lemma $\ref{newlem}$.  The only contributions are from the divisors $D\in \delta_{\infty}$ such that $C \cap D= \emptyset$, which each contribute a   $\Q(-1)$ to $E_1^{0,2}$ and nothing else.
\end{proof}
By way of example: on $\overline{\Mod}_{0,6}$, with the standard dihedral ordering, there is a unique irreducible boundary component which does not meet the standard cell,
namely the divisor defined by $\{135\} \cup \{ 246\}$. On $\overline{\Mod}_{0,5}$ there are none.

\subsection{Proof of theorem \ref{thmCohomVanish}} Let $\ell =n-3$.
The restriction map 
 $$ \gr^W_0 H^{\ell} (\overline{\Mod}_{0,n},B)       \overset{\sim}{\To}    \gr^W_0 H^{\ell}(\overline{\Mod}_{0,n}\backslash A, B\backslash (B \cap A)) $$
is an isomorphism by a relative version of the Gysin sequence. 
 Therefore $\gr^W_0 m(A,B) \cong \Q(0)$ follows from lemma $\ref{lemcompletebroken}$, and $\gr^W_{2\ell} m(A,B)_{dR} \cong \Q(-\ell)$
 follows by duality.
 
 Next, apply the relative cohomology spectral sequence $(\ref{RelCohomSS})$ with $X=\overline{\Mod}_{0,n}$. By general facts about mixed Hodge structures (or from $(\ref{GysinSS})$), we know that 
 $$ E_1^{p,q}   = \bigoplus_{|I|=p}  H^{q}( B_{I} \backslash (B_{I} \cap A))$$
 has weights in the interval $[q,2q]$.  Therefore  it suffices to prove that 
 $$\gr^W_{2\ell-2}  E_1^{0,\ell}   = 0 \qquad \hbox{ and } \qquad \gr^W_{2\ell-2}   E_1^{1,\ell-1} = 0 \ .$$
First of all, 
$$\gr^W_{2\ell -2}\, E_1^{0,\ell}  =  \gr^W_{2\ell-2} H^{\ell} (\overline{\Mod}_{0,n} \backslash A) \cong (\gr^W_2 H^{\ell} (\overline{\Mod}_{0,n} , A))^{\vee} $$
by duality. This vanishes for $\ell>2$ on application of lemma  \ref{lemgrw2vanishes}.  For $\ell=2$ ($n=5$), it vanishes from the second part of the same lemma and the remark following it.
 
 On the other hand,
 $$\gr^W_{2\ell -2}  E_1^{1,\ell- 1}  =     \bigoplus_i  \, \gr^W_{2\ell -2} \,  H^{\ell-1}(B_i \backslash (B_i \cap A)) \cong    \bigoplus_i  \, ( \gr^W_{0} \, m(B_i, B_i \cap A))^{\vee}  $$
 by duality. The summands on the right vanish by lemma $\ref{newlem}$, since $A$ is cellular.

\bibliographystyle{plain}
\bibliography{main}

\end{document}